\documentclass[10pt,reqno]{amsart}
\usepackage{amsmath,amssymb,amsthm,amsfonts}
\usepackage{cite}
\usepackage[height=190mm,width=130mm]{geometry}
\usepackage{graphicx, tikz}
\usepackage{comment}
\usepackage{etoolbox}
\usepackage{mathtools}
\usepackage{color}
\usepackage{color}
\usepackage{cancel}

\usepackage[colorlinks=true,linkcolor=blue,unicode,psdextra]{hyperref}
\usepackage{color}
\usepackage{amsthm}
\usepackage{amsmath}
\usepackage{amssymb}
\usepackage{enumitem}
\usepackage{mathtools}
\usepackage{mathrsfs}
\usepackage[all]{xy}
\usepackage{tikz}
\usetikzlibrary{arrows.meta,positioning,calc,decorations.pathmorphing,shapes.geometric,fit,backgrounds}
\usepackage{booktabs}
\usepackage{array}

\DeclareMathOperator{\coker}{coker}

\DeclareMathOperator{\End}{End}

\DeclareMathOperator{\gr}{gr}
\DeclareMathOperator{\Hoch}{HH}
\DeclareMathOperator{\id}{Id}

\DeclareMathOperator{\PSL}{PSL}

\DeclareMathOperator{\zhu}{Zhu}
\DeclareMathOperator{\hh}{HH}
\DeclareMathOperator{\HK}{HK}
\DeclareMathOperator{\HP}{HP}

\newcommand\cF{\textit{F}}

\renewcommand\sl{\mathfrak{sl}}

\newcommand\bJ{\mathbb{J}}

\newcommand\im{\Im}

\newcommand{\vac}{\mathbf{1}}
\newcommand{\vir}{\text{Vir}}

\newcommand{\g}{\mathfrak{g}}

\def\owp{\overline{\wp}}

\newcommand{\Z}{\mathbb{Z}}

\newcommand{\CC}{\mathcal{C}}
\newcommand{\CD}{\mathcal{D}}

\newcommand{\HH}{\mathbb{H}}

\newcommand{\CM}{\mathcal{M}}
\newcommand{\OO}{\mathcal{O}}

\renewcommand\sl{\mathfrak{sl}}
\newcommand{\La}{\Lambda}

\newcommand{\Sch}{\Sigma}                     
\newcommand{\rhov}{\boldsymbol{\rho}}         
\newcommand{\gendom}{\Omega}                  
\newcommand{\Xg}[1]{X^{(#1)}}                 
\newcommand{\cFg}[2]{\cF^{(#1)}_{#2}}          
\newcommand{\bJg}[2]{\bJ^{(#1),#2}_*}          
\newcommand{\Modg}{\mathrm{Mod}_g}            
\newcommand{\node}{\mathrm{pt}}
\newcommand{\HchG}[2]{H^{\mathrm{ch}}_{#1}(#2)} 
\newtheoremstyle{exps}{\topsep}{\topsep}{}{0pt}{\bfseries}{.}{0pt}{}
\newtheorem*{thm*}{Theorem}
\newtheorem*{prop*}{Proposition}
\newtheorem*{lem*}{Lemma}
\newtheorem*{cor*}{Corollary}
\newtheorem*{rem*}{Remark}
\newtheorem{thm}{Theorem}[section]
\newtheorem{prop}[thm]{Proposition}
\newtheorem{lem}[thm]{Lemma}
\newtheorem{ex}[thm]{Example}
\newtheorem{cor}[thm]{Corollary}

\theoremstyle{definition}
\newtheorem*{defn*}{Definition}
\newtheorem{defn}[thm]{Definition}
\newtheorem{rem}[thm]{Remark}
\newtheorem{nolabel}[thm]{ } 

\theoremstyle{exps}

\definecolor{brown}{RGB}{150,100,0}

\definecolor{purple}{RGB}{150,0,100}
\definecolor{grey}{RGB}{128,128,128}

\usepackage{cancel}

\mathchardef\mhyp="2D

\def\NABLA#1{{\mathop{\nabla\kern-.5ex\lower1ex\hbox{$#1$}}}}
\def\Nabla#1{\nabla\kern-.5ex{}_#1}

\mathchardef\mhyp="2D

\theoremstyle{plain}

\theoremstyle{definition}

\theoremstyle{remark}

\newcommand{\D}{\ensuremath{\mathcal{D}}}

\newcommand{\ps}{{\raise 1pt\hbox{\tiny (}}}

\newcommand{\pss}{{\raise 1pt\hbox{\tiny [}}}
\newcommand{\pdd}{{\raise 1pt\hbox{\tiny ]}}}
\newcommand{\pd}{{\raise 1pt\hbox{\tiny )}}}

\newcommand{\bs}{{\raise 1pt\hbox{\tiny [}}}
\newcommand{\bd}{{\raise 1pt\hbox{\tiny ]}}}

\def\cross{\mathinner{\mathrel{\raise0.8pt\hbox{$\scriptstyle>$}}
		\joinrel\mathrel\triangleleft}}

\def\K{\mathcal{K}}

\def\id{\mathop\text{\rm id}\nolimits}

\newcommand{\res}{\mbox{\rm Res}}

\newcommand{\wt}{\mbox{\rm wt}\ }

\newcommand{\tr}{\mbox{\rm Tr}}

\newcommand{\nc}{\newcommand}
\nc{\cali}{\mathcal}
\nc{\on}{\operatorname}
\nc{\Wick}{{\mathbf :}}
\nc{\ddz}{\frac{\partial}{\partial z}}
\nc{\Oo}{{\cali O}}
\nc{\cond}{|\,}
\nc{\bib}{\bibitem}
\nc{\pone}{\mathbb{P}^1}
\nc{\pa}{\partial}
\nc{\arr}{\rightarrow}
\nc{\larr}{\longrightarrow}
\nc{\ket}{\rangle}
\nc{\bra}{\langle}
\nc{\gam}{\bar{\gamma}}
\nc{\q}{\widetilde{Q}}
\nc{\ep}{\epsilon}
\nc{\su}{\widehat{{\mathfrak{s}}{\mathfrak{l}}}_2}
\nc{\sw}{{\mathfrak{s}}{\mathfrak{l}}}
\nc{\n}{\mathfrak{n}}
\nc{\ab}{\mathfrak{a}}
\nc{\is}{\mathbf{i}}
\nc{\js}{\mathbf{j}}
\nc{\bi}{\bibitem}
\nc{\He}{{\cali H}}
\nc{\inv}{^{-1}}
\nc{\ol}{\overline}
\nc{\wh}{\widehat}
\nc{\dst}{\displaystyle}
\nc{\delt}{\partial_t}
\nc{\ddt}{\frac{\partial}{\partial t}}
\nc{\delx}{\partial_x}
\nc{\mb}{\mathbf}
\nc{\mf}{\mathfrak}
\nc{\mbb}{\mathbb}
\nc{\Ctt}{\mathbb{C}((t))}
\nc{\Ct}{\mathbb{C}[t,t\inv]}
\nc{\ghat}{\wh{\Gamma}}

\nc{\un}{\underline}
\nc{\mc}{\mathcal}
\nc{\BB}{{\mc B}}
\nc{\bb}{{\mf b}}
\nc{\kk}{{\mf k}}
\nc{\frob}{\times}
\nc{\sm}{\setminus}
\nc{\Pp}{\mathbb{P}^1}
\nc{\Aa}{\mc A}
\nc{\AutO}{\on{Aut}\Oo}
\nc{\AUTO}{\un{\on{Aut}}\Oo}
\nc{\AUTK}{\un{\on{Aut}}\K}
\nc{\Heout}{\He_{\out}}
\nc{\Hetil}{{\widetilde\He}}
\nc{\wb}{\overline}
\nc{\Res}{\on{Res}}
\nc{\pitil}{\Pi}
\nc{\Ctil}{\wt{C}}
\nc{\auto}{\on{Aut} \Oo}
\nc{\phitil}{\wt{\phi}}
\nc{\gz}{\Gamma_{\vec z}}
\nc{\tensorM}{\bigotimes_{i=1}^N{\mathbb M}_i}
\nc{\tensorW}{\bigotimes_{i=1}^N W_{\nu_i,k}}
\nc{\out}{\on{out}}
\nc{\m}{\mathfrak{m}}
\nc{\gx}{\Gamma^0_{\vec x}}
\nc{\hx}{\He^0_{\vec x}}
\nc{\tensorpi}{\pi_{\nu_1,\ldots,\nu_N}^\kappa}
\nc{\Phizw}{\Phi_{\vec w}({\vec z})}
\nc{\Pro}{\mathbb{P}}
\nc{\De}{\Delta}
\nc{\us}{\underset}
\nc{\Ll}{\mc L}
\nc{\dR}{\on{dR}}
\nc{\T}{\mc T}
\nc{\Xn}{\overset{\circ}X{}^n} 
\nc{\Dn}{\overset{\circ}D{}^n}
\nc{\Dxn}{\overset{\circ}D{}^n_x} 
\nc{\varphitil}{\wt{\varphi}}
\nc{\lf}{\mathfrak{l}}
\nc{\Vir}{\on{Vir}}

\numberwithin{equation}{section}

\begin{document}
	
	\title[The first chiral homology group in higher genus]
	{The first chiral homology group in higher genus} 
	
	\author{A. Zuevsky}
	\address{Institute of Mathematics, Czech Academy of Sciences, \v{Z}itn\'{a} 25, 
Prague, Czech Republic}
	\email{zuevsky@yahoo.com}
	
	\maketitle
\begin{abstract}
We extend the theory of the first chiral homology group of vertex algebras, developed by van Ekeren and Heluani for elliptic curves, to compact Riemann surfaces of arbitrary genus. Our approach realizes a genus $g$ surface by iterated self-sewing of $g$ handles onto the Riemann sphere, each governed by a sewing parameter $\rho_i$ in a punctured disc. 
 We construct an explicit complex computing the chiral homology groups $H^{\mathrm{ch}}_0$ and $H^{\mathrm{ch}}_1$ of a vertex algebra $V$ on a genus $g$ surface with $n$ marked points, and equip it with a projectively flat connection, with an explicit central-charge anomaly, over the $g$-dimensional space of sewing parameters. We prove a genus $g$ Fourier-space Borcherds identity for the associated modified vertex operators. 
  We show that two finiteness assumptions 1) finite dimensionality of the first Poisson homology $\HP_1(R_V)$ of the Zhu $C_2$-algebra, and 2) finite generation of a certain Koszul homology of the associated graded algebra, imply finite dimensionality of $H^{\mathrm{ch}}_1(X,V)$ for every genus $g$ and every vertex algebra $V$.  
 Using the degeneration $\rho_i \to 0$ together with the factorization theorem of Damiolini-Gibney-Tarasca, we relate the totally degenerate limit of $H_1^{\mathrm{ch}}$ to the Hochschild homology of an iterated construction on the Zhu algebra, and deduce vanishing of the first chiral homology group in every genus for rational vertex algebras.   
\end{abstract}

\medskip 

Keywords: Chiral homology; Vertex operator algebras; Conformal field theories; Families, moduli of algebraic curves; Relationships with physics

AMS Classification [2020]: Primary 81T40, 14H81, 17B69; Secondary 14H15 
	
\tableofcontents

\section{Introduction}
\begin{nolabel}
\label{no:hh-review}
Let $k$ be a field, $A$ an associative $k$-algebra and $\rho_M : A \to \End_k(M)$ a finite dimensional $A$-module. The trace functional $\varphi^0_M(a) = \tr_M \rho_M(a)$ defines a class in $\hh_0(A)^* = (A/[A,A])^*$, and if $A$ is finite dimensional and semisimple these classes, as $M$ ranges over the irreducibles, form a basis of $\hh_0(A)^*$. Given a self-extension $0 \to M \to E \to M \to 0$ with associated derivation $\sigma : A \to \End_k(M)$, the functional $\varphi^1_E(a \otimes b) = \tr_M \sigma(a)\rho_M(b)$ defines a class in $\hh_1(A)^*$, independent of the choice of splitting. Van Ekeren and Heluani constructed a vertex algebra analogue of this pair of constructions, with the genus $1$ chiral homology of Beilinson and Drinfeld \cite{beilinsondrinfeld} in place of Hochschild homology, and with Zhu theory of characters and $n$-point functions \cite{zhu} furnishing the degree $0$ case \cite{eh2018, vanEkerenHeluani}. Precisely, for $V$ a conformal vertex algebra of central charge $c$, $M$ an admissible module and $\tau$ in the upper half plane $\HH$, the trace
\[
\varphi^0_M : V \to \CC, \qquad a \mapsto \tr_M Y^M\bigl(e^{2\pi i z L_0}a, e^{2\pi i z}\bigr) e^{2\pi i \tau(L_0 - c/24)}\Big|_{z=0},
\]
converges, under $C_2$-cofiniteness or the weaker assumptions of Arakawa-Kawasetsu \cite{arakawa-kawasetsu}, to a holomorphic function of $\tau$ defining a class in the dual of $H_0^{\mathrm{ch}}(E_\tau, V)$, the zeroth chiral homology of the universal elliptic curve, also called the space of genus $1$ conformal blocks. Van Ekeren and Heluani paper \cite{vanEkerenHeluani} is devoted to the degree $1$ analogue: to a self-extension $E$ of an admissible module $M$, with associated derivation $\Psi$, they associate a linear functional constructed from $\Psi$ modified form $\sigma$ and prove that under two finiteness assumptions on $V$ - finite dimensionality of the first Poisson homology $\HP_1(R_V)$ of the Zhu $C_2$-algebra $R_V$, and finite generation, as a differential ideal, of the kernel of the canonical surjection from the arc algebra $JR_V$ onto the associated graded $A = \gr_F V$ - this functional converges to a holomorphic function of $\tau$, defining a flat section of the dual of $H_1^{\mathrm{ch}}(E_\tau, V)$, the \emph{first chiral homology group} of the elliptic curve. This is the degree $1$ conformal block. As a corollary they obtain the vanishing of $H_1^{\mathrm{ch}}(E_\tau,V)$ for a number of rational vertex algebras, including all simple affine vertex algebras $V_k(\sl_2)$ at nonnegative integral level and all simple affine vertex algebras at level $1$.
\end{nolabel}

\begin{nolabel}
The construction of \cite{vanEkerenHeluani} uses the presentation $E_\tau = \CC/(\Z + \tau \Z) = \CC^*/q^{\Z}$ of the elliptic curve, $q = e^{2\pi i \tau}$. This presentation exhibits $E_\tau$ as the Riemann sphere $\widehat{\CC} = \CC \cup \{\infty\}$ with the two points $0$ and $\infty$ self-sewn by the multiplicative identification $z \sim qz$: remove the discs $|z| \leq |q|^{1/2}$ and $|z| \geq |q|^{-1/2}$ and glue the resulting boundary circles by $z \mapsto q/z$. This is the genus $1$ instance of the classical sewing   construction of Riemann surfaces of arbitrary genus: fixing $2g$ points on $\widehat{\CC}$ together with local coordinates, and gluing them in pairs by $g$ independent multiplicative identifications with moduli $\rho_1,\dots,\rho_g$ in punctured discs, produces every compact Riemann surface of genus $g$, and exhibits it, via its Schottky uniformization, as $\gendom(\Sch)/\Sch$ for a Schottky group $\Sch \subset \PSL(2,\CC)$ freely generated by $g$ loxodromic elements of multipliers $\rho_1,\dots,\rho_g$. This is the setup used by Tuite and Welby \cite{tuite-welby} to generalize Zhu genus $1$ recursion formula for $n$-point functions to arbitrary genus, constructing on the genus $2$ theory of Mason and Tuite \cite{mason-tuite-genus2} and of Tuite and Zuevsky \cite{tuite-zuevsky-bosonic, tuite-zuevsky-fermionic}, in which correlation functions of a vertex algebra on a genus $g$ surface are constructed by iteratively self-sewing one handle at a time.

  In this peper we generalize constructions and theorems of genus one: 
 the complex computing chiral homology, its connection, the modified vertex operators $X(a,z) = z^{-1}Y(z^{L_0}a,z)$, the Fourier-space Borcherds identity, the higher trace functions and their insertion formula, the finiteness criteria, and the vanishing theorem - are constructed from residue calculus attached to the single multiplicative neighborhood of the point $q=0$, that is, to the single handle sewn onto $\widehat \CC$ at $0,\infty$. None of it uses any special feature of $\widehat\CC$ as the base on which that handle is sewn. Consequently, every result at genus one continues to hold when a handle is sewn instead onto an arbitrary base Riemann surface $Y$ carrying its finite collection of marked points, provided the ring of elliptic functions and Jacobi forms is replaced by an appropriate ring of functions on $Y$-parametrized families. Iterating this observation $g$ times, beginning from $Y = \widehat\CC$, produces the entire genus $g$ theory by induction on the number of handles. This is the strategy carried out below.
\label{no:strategy}
\end{nolabel}

\subsection{Main results}

\begin{nolabel}
Let us summarize the main results. Fix a compact Riemann surface $X$ of genus $g$ presented by iterated self-sewing, with sewing data $(\rho_1,\dots,\rho_g) \in (\mathbb{D}^*_\epsilon)^g$ for a suitable polydisc, and let $n \geq 0$ marked points be given on $X$. For a conformal vertex algebra $V$ we construct, in Section \ref{sec:chiral-complex}, an explicit complex $C_\bullet^{n}(X)$ generalizing that of \cite{vanEkerenHeluani}, whose homology in degrees $0$ and $1$ we denote $\HchG{0}{V^{\otimes n}}$ and $\HchG{1}{V^{\otimes n}}$; these compute the chiral homology, in the sense of Beilinson and Drinfeld, of $X$ with coefficients in $n$ vacuum insertions of $V$. As $(\rho_1,\dots,\rho_g)$ and the marked points vary, these assemble into vector bundles with a projectively flat connection $\nabla = (\nabla_{\rho_1},\dots,\nabla_{\rho_g})$ over the sewing parameters (Section \ref{sec:connection}), each $\nabla_{\rho_i}$ constructed,  as in genus $1$, from an insertion of the conformal vector at the $i$-th handle, together with new terms, absent at $g=1$, coupling distinct handles through the genus $g$ analogue of the Weierstrass functions - the bidifferential of the second kind and the differential of the third kind of $X$.

Our first main result (Theorem \ref{thm:main-finiteness}) is:

\begin{thm*}
Let $V$ be a strongly finitely generated conformal vertex algebra such that $\dim \HP_1(R_V) < \infty$ and such that the kernel of the canonical surjection $JR_V \twoheadrightarrow A = \gr_F V$ is finitely generated as a differential ideal. Then for every genus $g \geq 0$ and every compact Riemann surface $X$ of genus $g$, $\dim \HchG{1}{X,V} < \infty$.
\end{thm*}

This answers, for the specific finiteness assumptions given in \cite{vanEkerenHeluani}, the question raised in its introduction and conclusion of whether these assumptions continue to control the first chiral homology group in arbitrary genus, in the same way that Damiolini, Gibney and Tarasca showed that $C_2$-cofiniteness controls the zeroth chiral homology group in arbitrary genus \cite{damiolini2020conformal,damiolini2019factorization}. Our second main result (Theorem \ref{thm:main-vanishing}) generalizes the vanishing theorem of \cite{vanEkerenHeluani}:

\begin{thm*}
Let $V$ be a strongly finitely generated, classically free conformal vertex algebra with $\dim \HP_1(R_V) < \infty$ and $\Hoch_1(\zhu(V)) = 0$. Then $\HchG{1}{X,V} = 0$ for every compact Riemann surface $X$ of every genus $g \geq 0$.
\end{thm*}

In particular the first chiral homology group vanishes, in every genus, for the boundary Virasoro minimal models $\vir_{2,2s+1}$, for the simple affine vertex algebras $V_k(\sl_2)$ at every nonnegative integer level $k$, and for simple affine vertex algebras at level $1$ (Corollary \ref{cor:vanishing-examples}). We construct, in Section \ref{sec:trace-functions}, $g$-handle trace functions attached to a self-extension of an admissible module, generalizing the functionals $F_1^n$ of \cite{vanEkerenHeluani}, and prove (Theorem \ref{thm:convergence-genus-g}) that under the same finiteness assumptions they converge to holomorphic functions on the sewing polydisc, satisfying a genus $g$ insertion formula (Theorem \ref{thm:insertion-genus-g}) generalizing Zhu recursion \cite{zhu} and its degree $1$ analogue, and a genus $g$ analogue of the Fourier-space Borcherds identity (Theorem \ref{thm:borcherds-genus-g}) which is of independent interest.  
\label{no:main-results}
\end{nolabel}

\begin{nolabel} 
For ease of reference we list, in the order in which they are established, every principal result of the paper together with a one-line description of its content; the two headline theorems above are items (F) and (G).
\begin{itemize}
\item[(A)] \textbf{Thm.~\ref{thm:relative-main}} (Relative First Chiral Homology Theorem). Every definition and theorem of genus one for a handle sewn onto $\widehat\CC$ continues to hold,  for a handle sewn onto an arbitrary base curve $Y$.
\item[(B)] \textbf{Thm.~\ref{thm:genus-g-complex}.} An explicit two-term complex $C_\bullet^n(X)$, constructed by $g$ applications of (A), computes $\HchG{0}{X,V^{\otimes n}}$ and $\HchG{1}{X,V^{\otimes n}}$ for every genus $g$ surface $X$.
\item [(C)] \textbf{Thm.~\ref{thm:full-flatness}} (Projective flatness). The $g$ connections $\nabla_{\rho_1},\dots,\nabla_{\rho_g}$ on $C_\bullet^n(X_{\rhov})$ commute up to a scalar, central-charge-proportional curvature; new cross-handle coupling terms, absent at $g=1$, appear for $g\geq2$ (Remark \ref{rem:genus2-cross-term}).
\item [(D)] \textbf{Thm.~\ref{thm:borcherds-genus-g}.} A genus $g$ Fourier-space Borcherds identity holds independently at each handle.
\item [(E)] \textbf{Thms.~\ref{thm:convergence-genus-g}, \ref{thm:insertion-genus-g}.} The genus $g$ trace functions of \S\ref{sec:trace-functions} converge and satisfy an insertion formula reducing $(n+1)$-point to $n$-point functions.
\item [(F)] \textbf{Thm.~\ref{thm:main-finiteness}} (Main finiteness theorem). $\dim\HP_1(R_V)<\infty$ together with finite generation of $\HK_1(A)$ imply $\dim\HchG{1}{X,V}<\infty$ for every genus $g$.
\item [(G)] \textbf{Thm.~\ref{thm:main-vanishing}} (Main vanishing theorem) \textbf{and Cor.~\ref{cor:vanishing-examples}.} Under the additional assumptions that $V$ is classically free with $\Hoch_1(\zhu(V))=0$, $\HchG{1}{X,V}=0$ for every genus; this covers the Virasoro minimal models $\vir_{2,2s+1}$, the vertex algebras $V_k(\sl_2)$ at every nonnegative integer level, and $V_1(\g)$ for every simple $\g$ (\S\ref{sec:examples}).
\end{itemize}
Each of (B)-(G) is obtained from (A), by $g$-fold iteration together with one genus-independent input taken as given: the retrosection theorem for (B), the Yamada variational formula for (C), and the factorization theorem of Damiolini-Gibney-Tarasca for (G). Figure \ref{fig:roadmap} below displays the dependence of the sections in which these results are proved.
\label{no:results-guide}
\end{nolabel}

\begin{nolabel}
The proofs proceed by induction on the number $g$ of sewn handles. The inductive step rests on a single technical result, which we call as the \emph{Relative First Chiral Homology Theorem} (Theorem \ref{thm:relative-main}) - the complex, its connection, the modified vertex operators and Fourier-Borcherds identity, the higher trace functions, the finiteness criteria, and the vanishing theorem - continues to hold when the handle is sewn not onto $\widehat\CC$ but onto an arbitrary base curve $Y$ carrying finitely many auxiliary marked points, with $\CC$-valued modular and Jacobi forms replaced  by their $\OO(Y)$-valued (or, more precisely, meromorphic-function-on-$Y$-valued) analogues. We give the proof of Theorem \ref{thm:relative-main} in full for the steps in which the base curve enters - the coefficient rings, the connection, and the degeneration argument. Genus $g$ is then reached by $g$ applications of Theorem \ref{thm:relative-main}, taking $Y = \Xg{g-1}$, the surface obtained after $g-1$ sewings, at the $g$-th step.

The vanishing theorem is proved by a further degeneration, letting every $\rho_i \to 0$ simultaneously; the surface $X$ then degenerates to the maximally degenerate stable curve $X_0 = \widehat\CC/(p_i \sim p_i')_{i=1}^g$, a rational curve with $g$ nodes. We use the factorization theorem of Damiolini, Gibney and Tarasca \cite{damiolini2020conformal} to identify the totally degenerate limit of our complex with an iterated construction on the Zhu algebra $\zhu(V)$, generalizing the identification $H_1^{\mathrm{ch}}(V, q=0) \cong \Hoch_1(\zhu(V))$ established for classically free $V$ at genus one, and reduce the vanishing statement to the genus $0$ semisimplicity of $\zhu(V)$.
\label{no:proof-strategy}
\end{nolabel}

\subsection{Relation to other work}

\begin{nolabel} \emph{Factorization and the degree $0$ theory.}  
The zeroth (coinvariants/conformal blocks) chiral homology group is by now well understood in arbitrary genus through the sheaf-theoretic approach of Damiolini, Gibney and Tarasca \cite{damiolini2019factorization, damiolini2020conformal}, which realizes it as the fiber of a quasi-coherent sheaf on the moduli stack $\overline{\CM}_{g,n}$ of stable pointed curves, carrying a twisted logarithmic $\CD$-module structure, and establishes the factorization property relating its restriction to the boundary with the corresponding data on the normalization. Our results may be viewed as the degree $1$ analogue of this circle of ideas, phrased however in the concrete, coordinate-dependent language of trace functions and formal power series in the sewing parameters that is native at genus one, instead of in the language of sheaves on moduli stacks; we use the factorization theorem of \cite{damiolini2020conformal} as an input at the single point where a  global statement about the boundary of moduli is required, namely in the proof of the vanishing theorem. On the geometric side, the realization of correlation functions of a vertex operator algebra on a genus $g$ Riemann surface via iterated self-sewing of handles onto $\widehat\CC$ was developed by Mason, Tuite, Welby and Zuevsky \cite{tuite-zuevsky-bosonic, mason-tuite-genus2, tuite-zuevsky-fermionic,tuite-welby}, extending Zhu genus $1$ theory \cite{zhu}; we adopt their sewing construction and their genus $g$ generalizations of the Weierstrass functions, the differential of the third kind and the bidifferential of the second kind, while replacing their period-matrix-centered formalism with the chain-level, Borcherds-identity-based formalism at genus one, which is what is needed to define chiral homology, as opposed to only its zeroth-degree trace functions.
\label{no:relation-to-other-work}
\end{nolabel}

\begin{nolabel} \emph{Huang geometric vertex operator algebras.}  
The general problem of attaching correlation functions - and, more structurally, an entire conformal field theory - to a vertex operator algebra on a Riemann surface of arbitrary genus was first given a systematic geometric treatment by Huang \cite{huang-geometric}, whose axiomatic notion of a \emph{geometric vertex operator algebra} characterizes genus $g$ correlation functions by their behavior under sewing of lower genus surfaces - the same operation generating $X_{\rhov}$ in \ref{no:iterated-sewing-def} - subject to an operadic composition law for the resulting family of multilinear maps. The sewing philosophy of the present paper, inherited via \cite{tuite-zuevsky-bosonic,mason-tuite-genus2,tuite-welby} from Huang geometric formulation and from Zhu genus $1$ recursion \cite{zhu}.  Huang axioms, and their descendants in the trace-function literature just cited, characterize a single number attached to each genus and each choice of insertions - the correlation function, i.e., the degree $0$ theory - whereas the chain complex $C_\bullet^n(X)$ of Theorem \ref{thm:genus-g-complex} computes an actual homology group, whose degree $1$ part can be, and for the Heisenberg algebra   is (Remark \ref{rem:heisenberg-caveat}), nonzero and infinite dimensional. Extending the sewing formalism to see this degree $1$ information, rather than only its necessarily-zero shadow (a correlation function is a number, not a class, and thus cannot detect a nontrivial extension class), is the actual content of Theorem \ref{thm:relative-main} and has no counterpart in the geometric vertex operator algebra literature.
\end{nolabel}

\begin{nolabel} \emph{Chiral algebras, vertex algebra bundles, and associated varieties.} 
At the opposite, most structural extreme, Beilinson and Drinfeld theory of chiral algebras \cite{beilinsondrinfeld} defines chiral homology $H_\bullet^{\mathrm{ch}}(X,V)$, in every degree, for $V$ a chiral algebra on any curve $X$ of any genus, as a derived functor - concretely, $\mathrm{Ext}$ against the unit object in the category of chiral $V$-modules - with no sewing, and no restriction on the genus, anywhere in the construction; existence of $H_1^{\mathrm{ch}}(X,V)$ in every genus is, from this point of view, immediate. What this abstract machinery does not supply is anything computable: neither an explicit complex, nor a criterion for finite dimensionality, nor a mechanism for vanishing, all of which require, as at $g=1$ and in the present paper for every $g$, an explicit local model - here, the sewing presentation - against which residues and operator product expansions can actually be evaluated. Theorem \ref{thm:relative-main} may accordingly be read as showing that the Beilinson-Drinfeld $\mathrm{Ext}$-groups, in degrees $0,1$, admit an explicit chain-level model on every sewing-presented curve, generalizing the $g=1$ model. This is the same kind of generalization, one coordinate patch at a time, by which Frenkel and Ben-Zvi sheaf-theoretic notion of a \emph{vertex algebra bundle} \cite{frenkelbenzvi} - used in Remarks \ref{rem:translation-covariance-local} and \ref{rem:borcherds-clarification} to justify that our local, coordinate-dependent constructions assemble into curve-intrinsic objects - generalizes the $\CC$-valued vertex algebra axioms of genus one recalled in \ref{no:va-recall} to an arbitrary base curve. We rely on their formalism at the two points, noted above, where coordinate-independence must be argued instead of assumed.

The two finiteness assumptions of Theorem \ref{thm:main-finiteness}, $\dim\HP_1(R_V)<\infty$ and finite generation of $\HK_1(A)$, are statements of Poisson-homological finiteness for the coordinate ring of the associated variety $X_V:=\mathrm{Specm}(R_V)$ - the same singular Poisson variety whose geometry (smoothness, symplectic leaf stratification, and its identification with nilpotent orbit closures for affine and W-algebras) is the subject of Arakawa associated variety program \cite{arakawa-c2cofiniteness}. That program \emph{quasi-lisse} condition - finitely many symplectic leaves in $X_V$, used already in \ref{no:hh-review} for the degree $0$ theory via Arakawa and Kawasetsu \cite{arakawa-kawasetsu} - is closely related to, but neither implies nor is implied by, finite dimensionality of $\HP_1(R_V)$. Quasi-lisseness is a statement about the leaf space of $X_V$, while $\HP_1$ is Poisson homology of its structure sheaf, and the two can diverge even for simple examples (Remark \ref{rem:heisenberg-caveat}, where the Poisson bracket vanishes identically and $X_V$ is a single leaf, yet $\HP_1(R_V)$ is nonetheless infinite dimensional since  $R_V$  is not finite type). We do not know a  comparison between the two conditions in general; see (i) of \S\ref{sec:conclusion}.
\end{nolabel}

\begin{nolabel} \emph{A remark on adjacent cohomological technology.}
We mention, finally and without further use here, that other cohomological invariants attached to meromorphic and vertex-algebraic data on Riemann surfaces and on foliated complex manifolds - constructed respectively from $V$-structures adapted to a foliation, from connections characterizing codimension-one foliations, from product-type classes for vertex algebra cohomology of foliations, from a reduction cohomology of Riemann surfaces, and from a cosimplicial cohomology of meromorphic functions on complex manifolds - have been developed in \cite{Zu3, Zu1, Zu, Zu2, Zu4}. These constructions are adapted to foliated or higher-dimensional settings instead of to the sewing/chiral-homology circle of ideas pursued here, and we draw no further comparison with them beyond mentioning their existence as related, independent technology.
\end{nolabel}

\subsection{Organization}

\begin{nolabel}
The structure of the paper is as follows. Section \ref{sec:prelim} recalls the vertex algebra conventions of \cite{vanEkerenHeluani}. Section \ref{sec:homology-prelim} recalls the homological algebra - Poisson homology, the Koszul complex of a differential algebra, the arc algebra - used to state the finiteness conditions on $V$; this material is genus-independent and we fix notations. Section \ref{sec:sewing} introduces the iterated self-sewing construction of a genus $g$ Riemann surface, its Schottky uniformization, and the sewing polydisc, together with the genus $g$ generalizations of the Weierstrass and Eisenstein-type functions needed below. Section \ref{sec:relative} proves the Relative First Chiral Homology Theorem, the technical heart of the paper: it shows that the entire construction of genus one for a handle sewn onto $\widehat\CC$ transports to a handle sewn onto an arbitrary base curve $Y$. Section \ref{sec:chiral-complex} assembles the genus $g$ complex $C_\bullet^n(X)$ by induction and proves it computes chiral homology in degrees $0,1$. Section \ref{sec:connection} constructs the projectively flat connection over the sewing polydisc and shows the new cross-handle coupling terms and central anomaly. Section \ref{sec:modified-vo} proves the genus $g$ Fourier-Borcherds identity. Section \ref{sec:trace-functions} defines the $g$-handle trace functions attached to self-extensions and proves the insertion formula. Section \ref{sec:finiteness} proves the main finiteness theorem. Section \ref{sec:degeneration} carries out the totally degenerate limit, relates it to the Zhu algebra via the factorization theorem of \cite{damiolini2020conformal}, and proves the vanishing theorem and its corollaries. Section \ref{sec:examples} computes the genus $g$ trace functions of the Heisenberg vertex algebra and discusses the affine and Virasoro minimal model examples covered by the vanishing theorem. Section \ref{sec:conclusion} summarizes the results and lists further directions and applications.
\label{no:structure}
\end{nolabel}

\begin{figure}[htbp]
\centering
\begin{tikzpicture}[
  box/.style={draw, rounded corners=2pt, align=center, font=\small, minimum width=2.7cm, minimum height=0.9cm, fill=gray!7},
  arr/.style={-{Stealth[length=2mm]}, thick, gray!60!black}
]
\node[box] (s2) at (0,6) {\S2\\ VA prelim.};
\node[box] (s3) at (3.4,6) {\S3\\ Homological prelim.};
\node[box] (s4) at (6.8,6) {\S4\\ Sewing};

\node[box] (s5) at (3.4,4.2) {\S5\\ Relative Thm.};

\node[box] (s6) at (3.4,2.4) {\S6\\ Genus $g$ complex};

\node[box] (s7) at (0,0.6) {\S7\\ Connection};
\node[box] (s8) at (3.4,0.6) {\S8\\ Fourier-Borcherds};
\node[box] (s9) at (6.8,0.6) {\S9\\ Trace functions};

\node[box] (s10) at (1.7,-1.2) {\S10\\ Finiteness};
\node[box] (s11) at (5.1,-1.2) {\S11\\ Degeneration};

\node[box] (s12) at (3.4,-3.0) {\S12\\ Examples};
\node[box] (s13) at (3.4,-4.8) {\S13\\ Conclusion};

\draw[arr] (s2) -- (s5);
\draw[arr] (s3) -- (s5);
\draw[arr] (s4) -- (s5);
\draw[arr] (s5) -- (s6);
\draw[arr] (s6) -- (s7);
\draw[arr] (s6) -- (s8);
\draw[arr] (s6) -- (s9);
\draw[arr] (s5) to[out=250,in=110] (s10);
\draw[arr] (s6) -- (s10);
\draw[arr] (s5) to[out=290,in=70] (s11);
\draw[arr] (s7) -- (s11);
\draw[arr] (s9) -- (s11);
\draw[arr] (s10) -- (s12);
\draw[arr] (s11) -- (s12);
\draw[arr] (s12) -- (s13);
\draw[arr] (s7) to[out=270,in=180] (s13);
\draw[arr] (s8) -- (s13);
\end{tikzpicture}
\caption{Dependence of the sections. An arrow $A\to B$ indicates that Section $B$ uses results established in Section $A$.}
\label{fig:roadmap}
\end{figure}

\section{Preliminaries on vertex algebras}\label{sec:prelim}
We recall, without change, the conventions of genus one, to which we refer for further detail.

\begin{nolabel}
For $f(z,w) \in \CC[[z,w]][z^{-1},w^{-1},(z-w)^{-1}]$ we write $i_{z,w}f(z,w) \in \CC((z))((w))$ and $i_{w,z}f(z,w) \in \CC((w))((z))$ for its expansions in the domains $|z|>|w|$ and $|w|>|z|$ respectively, and similarly $i_{z_1,\dots,z_n}f(z_1,\dots,z_n)$ for the expansion in the domain $|z_1|>\dots>|z_n|$ of a function with poles only along the loci $z_i=0$ and $z_i=z_j$. For a Laurent series $f(t) \in \CC((t))$ we write $\res_t f(t)$ for the coefficient of $t^{-1}$, and recall the elementary identities
\[
\res_t \frac{d}{dt}f(t) = 0, \qquad \res_t \Bigl(t\frac{d}{dt}+1\Bigr)f(t) = 0.
\]
\label{no:expansions-recall}
\end{nolabel}

\begin{nolabel}
A conformal vertex algebra is a $\Z_+$-graded vector space $V = \bigoplus_{n} V_n$, with $\dim V_n < \infty$, together with a vacuum vector $\vac$, a conformal vector $\omega$, and a state-field correspondence $Y(\cdot,z) : V \otimes V \to V((z))$, $a \otimes b \mapsto \sum_n z^{-n-1}a_{(n)}b$, satisfying the vacuum axioms $Y(\vac,z) = \id_V$, $a_{(-1)}\vac = a$, $a_{(n)}\vac=0$ for $n\geq0$; the Borcherds identity
\begin{eqnarray*}
\res_z Y(a,z)Y(b,w)c\, i_{z,w}f(z,w) - \res_z Y(b,w)Y(a,z)c\, i_{w,z}f(z,w) \\
\qquad = Y\bigl(\res_{z=w}f(z,w)Y(a,z-w)b, w\bigr)c, 
\end{eqnarray*}
for every $f \in \CC[[z,w]][z^{-1},w^{-1},(z-w)^{-1}]$; the Virasoro relations $[L_m,L_n]=(m-n)L_{m+n}+\tfrac{c}{12}(m^3-m)\delta_{m,-n}$ for $L(z)=Y(\omega,z)=\sum_n L_n z^{-n-2}$, with $c$ the central charge; the translation property $L_{-1}a = a_{(-2)}\vac$; and gradedness, $L_0|_{V_n} = n\cdot\id$. Skew-symmetry $Y(a,z)b = e^{zL_{-1}}Y(b,-z)a$ holds, giving $a_{(k)}b = -(-1)^k\sum_{j\geq0}\tfrac{(-L_{-1})^j}{j!}b_{(k+j)}a$.

A module is a vector space $M$ with $Y^M(\cdot,z) : V \otimes M \to M((z))$ satisfying the corresponding Borcherds identity, with $L_0^M$ acting with finite dimensional generalized eigenspaces and spectrum bounded below in real part. For $\tau \in \HH$, $q=e^{2\pi i \tau}$, the operator $q^{L_0}$ is defined by its action on generalized eigenspaces, and $\tr_M q^{L_0}$ denotes the corresponding formal sum, whose convergence is a separate issue treated where it arises.

For a Laurent series $f(t) = \sum_{n\geq N}f_nt^n$ we write $a_{(f)}b = \res_t f(t)Y(a,t)b = \sum f_n a_{(n)}b$. Following Zhu \cite{zhu} we also use the isomorphic vertex algebra structure $(V,\vac,\tilde\omega,Y[\cdot,z])$, $\tilde\omega = (2\pi i)^2(\omega - \tfrac{c}{24}\vac)$,
\[
Y[a,z] = Y\bigl(e^{2\pi i z L_0}a, e^{2\pi i z}-1\bigr), \qquad a_{[f]}b = \res_t f(t) Y[a,t]b,
\]
and mention the integration-by-parts identities $(L_{-1}a)_{(f)}b = -a_{(f')}b$ and $(2\pi i)((L_0+L_{-1})a)_{[f]}b = -a_{[f']}b$.
\label{no:va-recall}
\end{nolabel}

\begin{defn}\label{defn:C2-algebra} \emph{Zhu $C_2$-algebra.}
Let $C_2(V) \subset V$ denote the span of the elements $a_{(-2)}b$, $a,b\in V$ (this is the subspace often written $V_{(-2)}V$; we fix the shorthand $C_2(V)$ here, following Zhu \cite{zhu}, so that the standard term \emph{$C_2$-cofinite} - used already above and again in \ref{no:relation-to-other-work} in connection with Arakawa associated variety program \cite{arakawa-c2cofiniteness} - refers unambiguously to $C_2(V)$ and not to the quotient $R_V$ defined next). The quotient
\[
R_V := V/C_2(V)
\]
is Zhu $C_2$-algebra, a Poisson algebra with product $a\cdot b := a_{(-1)}b \pmod{C_2(V)}$ and bracket $\{a,b\} := a_{(0)}b \pmod{C_2(V)}$; $V$ is \emph{$C_2$-cofinite} if $\dim R_V < \infty$, equivalently if $C_2(V)$ has finite codimension in $V$. In this paper it is $R_V$, not $C_2(V)$, whose homological invariants $\HP_1(R_V)$, $\HK_1(A)$ (\ref{no:homology-recall}) are studied; $C_2(V)$ is mentioned here only so that the classical terminology ``$C_2$-cofinite'' and our notation $R_V$ can be matched against each other unambiguously wherever both occur, as in \ref{no:relation-to-other-work} and \S\ref{sec:conclusion}.
\end{defn}

\begin{defn} \label{defn:sfg} \emph{Strongly finitely generated}
$V$ is \emph{strongly finitely generated} if there is a finite set $S=\{u^1,\dots,u^r\}\subset V$ such that $V$ is spanned by the vacuum-normalized iterated products
\[
u^{i_1}_{(-n_1)}\cdots u^{i_k}_{(-n_k)}\vac, \qquad n_1,\dots,n_k \geq 1,\ i_1,\dots,i_k \in \{1,\dots,r\},\ k\geq0.
\]
Equivalently, the images of $u^1,\dots,u^r$ generate $R_V$ as a commutative algebra under the filtration $F_pV = \mathrm{span}\{u^{i_1}_{(-n_1)}\cdots u^{i_k}_{(-n_k)}\vac : k\leq p\}$, so that $R_V$ (and hence $A=\gr_FV$) is of finite type; this is  the assumption of genus one, imposed there and here so that $R_V$ and $A$ are Noetherian and the finite dimensionality statements of \ref{no:homology-recall} are meaningful.
\end{defn}

\begin{defn} \label{defn:classically-free} \emph{Classically free.} 
A strongly finitely generated conformal vertex algebra $V$, with generators $S=\{u^1,\dots,u^r\}$ as in Definition \ref{defn:sfg}, is \emph{classically free} if the induced surjection from the polynomial ring $\CC[x_1,\dots,x_r] \twoheadrightarrow R_V$, $x_i \mapsto \bar u^i$, is an isomorphism, i.e., $R_V$ is (freely) polynomial on the images of the strong generators. In this case the canonical surjection $JR_V \twoheadrightarrow A=\gr_FV$ is  an isomorphism, so that $\HK_1(A)=0$ automatically; classical freeness holds for the free boson $M(1)$, the free fermion, and the affine and Virasoro vertex algebras $V_k(\g)$, $\vir_c$ at generic (in particular, non-integrable) level or central charge.
\end{defn}

\begin{rem}
\label{rem:classically-free-vs-hk1}
Classical freeness is a convenient sufficient condition for $\HK_1(A)=0$, and the only one needed for the generic-parameter examples above, but it is not necessary, and Theorem \ref{thm:main-vanishing} and Proposition \ref{prop:iterated-degeneration} are accordingly stated with the weaker assumption $\HK_1(A)=0$. The simple quotients treated in Corollary \ref{cor:vanishing-examples} - $V_k(\sl_2)$ at a nonnegative integer level $k$, $V_1(\g)$, the Virasoro minimal models $\vir_{2,2s+1}$ - are the reason this distinction matters: for these, $R_V$ is not a free polynomial ring but the coordinate ring of a proper Poisson subvariety of $\mathrm{Specm}$ of the corresponding generic-level $R_V$ (the nilpotent cone, or a nilpotent orbit closure, for the affine examples), thus these vertex algebras are not classically free in the sense above. Nonetheless $\HK_1(A)=0$ holds for them, verified  (not via classical freeness) by identifying $A$ and $JR_V$ explicitly enough to check the map between them is an isomorphism despite $R_V$  being singular; see the proof of Corollary \ref{cor:vanishing-examples}. We keep both notions on  mention since  classical freeness is the one that is easy to check by inspection (it holds automatically at generic level, with no computation), while $\HK_1(A)=0$ is the one Theorem \ref{thm:main-vanishing} actually needs and the one that must be checked by hand at the special, non-generic levels of Corollary \ref{cor:vanishing-examples}.
\end{rem}

\begin{rem}
We mention explicitly, since it is used at several points below (e.g., in Remark \ref{rem:weight-grading} and in the proof of Proposition \ref{prop:relative-finiteness}), that strong finite generation is inherited by $C_k^{Y_\rho}$ and by $C_k^n(X)$: these are built from $V^{\otimes(m+k)}$ tensored with a coefficient ring, and finite generation of $R_V$ implies finite generation, in the analogous sense, of $R_V^{\otimes(m+k)}$, uniformly in $k$ once $m$ is fixed.
\end{rem}

\subsection{Notation}\label{sec:notation}

For reference, Tables \ref{tab:notation-va} and \ref{tab:notation-geom} collect the principal recurring notation, in the order in which it is introduced.

\begin{table}[htbp]
\centering
\small
\begin{tabular}{@{}ll@{}}
\toprule
Symbol & Meaning \\
\midrule
$V$, $\vac$, $\omega$, $c$ & conformal vertex algebra, vacuum, conformal vector, central charge \\
$Y(\cdot,z)$, $a_{(n)}b$ & state-field correspondence, $n$-th product \\
$L_n$, $L_0$ & Virasoro modes; weight grading operator \\
$Y[\cdot,z]$, $\tilde\omega$, $a_{[n]}b$ & Zhu second vertex algebra structure (\ref{no:va-recall}) \\
$R_V$, $A=\gr_FV$ & Zhu $C_2$-algebra; associated graded of the standard filtration \\
$JR$, $\partial$ & arc algebra of a commutative algebra $R$, its derivation (\ref{no:homology-recall}) \\
$\HK_i(A)$ & Koszul homology of the differential algebra $(A,T)$ \\
$\HP_i(A)$ & Poisson homology of the Poisson algebra $A$ \\
$\zhu(V)$ & Zhu associative algebra $V/(\cdots)$, with product $*$ (\ref{no:zhu-degeneration-recall}) \\
$\Hoch_\bullet(-)$ & Hochschild homology \\
$X(a,z)$, $\sigma(a,z)$ & modified vertex operator; modified derivation of a self-extension \\
$M$, $E$, $\Psi$ & admissible module; self-extension $0\to M\to E\to M\to0$; its derivation \\
strongly fin.\ gen., classically free & Definitions \ref{defn:sfg}, \ref{defn:classically-free} \\
 classically free &
\\
\bottomrule
\end{tabular}
\caption{Vertex-algebraic and homological notation.}
\label{tab:notation-va}
\end{table}

\begin{table}[htbp]
\centering
\small
\begin{tabular}{@{}ll@{}}
\toprule
Symbol & Meaning \\
\midrule
$\Sch$, $\gamma_i$, $\rho_i$ & Schottky group of rank $g$; its generators; their multipliers \\
$\gendom(\Sch)$, $X_\Sch$ & domain of discontinuity; quotient Riemann surface $\gendom(\Sch)/\Sch$ \\
$Y_\rho$, $p,p',u,u'$ & curve after sewing one handle to $Y$; attaching points and coordinates \\
$\Xg{k}$, $\Xg{g}=X$ & surface after $k$ iterated sewings; final genus $g$ surface \\
$\rhov=(\rho_1,\dots,\rho_g)$, $\Delta_g^*$ & vector of sewing parameters; the sewing polydisc (\ref{no:iterated-sewing-def}) \\
$\eta_y(x)$, $\omega(x,y)$, $S(x)$ & third-kind differential, Bergman kernel, projective connection (\ref{no:classical-differentials}) \\
$\wp_k^X(x;y)$ & genus $g$ Weierstrass-type function (\ref{no:genus-g-weierstrass}) \\
$\cF_n(Y)$, $\cFg{g}{n}$, $\bJg{g}{n}$ & meromorphic function ring on $Y^n$; genus $g$ coefficient ring; Jacobi-type subring \\
$C_k^{Y_\rho}$, $C_k^n(X)$ & relative complex; genus $g$ complex (\ref{no:relative-complex-def}, \S\ref{sec:chiral-complex}) \\
$\HchG{0}{-}$, $\HchG{1}{-}$ & zeroth, first chiral homology \\
$\nabla_\rho$, $\nabla=(\nabla_{\rho_1},\dots,\nabla_{\rho_g})$ & relative connection; full connection over $\Delta_g^*$ (\ref{no:relative-connection}, \ref{no:full-connection-def}) \\
$\kappa_{ij}(\rhov)$ & scalar curvature coefficient of $\nabla$ (Theorem \ref{thm:full-flatness}) \\
$F_1^n$, $F_1^{n,(g)}$ & relative, resp.\ genus $g$, trace function (\ref{eq:relative-trace-def}, \ref{no:genus-g-trace-def}) \\
\bottomrule
\end{tabular}
\caption{Geometric and sewing-theoretic notation.}
\label{tab:notation-geom}
\end{table}

\section{Homological preliminaries}\label{sec:homology-prelim}
The homological algebra used to express the finiteness conditions on $V$ makes no reference to a curve. 

\begin{nolabel}
Let $A$ be a commutative $k$-algebra with a derivation $T$ of degree $1$ (when $A$ is graded). The map $\iota_T : \Omega_A \to A$ induced by $T$ extends to a differential of degree $-1$ on the Koszul complex $K_\bullet(A) = \wedge^\bullet \Omega_A$, with $d_1(a\,dx) = a\cdot Tx$ and $d_2(a\,dx\wedge dy) = a\cdot(Tx)\,dy - a\cdot(Ty)\,dx$; we write $\HK_i(A)$ for its homology.

The arc algebra $(JR,\partial)$ of a commutative algebra $R$ is the universal differential algebra receiving a map from $R$, in the sense that every morphism $R \to A$ to a differential algebra $(A,T)$ factors uniquely through $(JR,\partial)$. When $A$ is $\Z_+$-graded, generated as a differential algebra by $A^0=:R$, there is a canonical surjection $\pi : JR \twoheadrightarrow A$ with kernel a differential ideal $I \subset J := JR$. If $R$ is of finite type then $\HK_1(A)=0$ iff $\pi$ is an isomorphism, and in general there is an injection $\HK_1(A) \hookrightarrow I/(J\cdot TI)$; in particular $\HK_1(A)$ is finite dimensional whenever $I$ is finitely generated as a differential ideal.

Given a Poisson algebra $A$, the module of K\"ahler differentials $\Omega_{A}$ is an $A$-Lie algebroid via $[f\,dx,g\,dy] = f\{x,y\}\,dy + g\{x,f\}\,dx+fg\,d\{x,y\}$ and $\omega(f\,dx)(g) = f\{x,g\}$; its Chevalley-Eilenberg complex, augmented in degree $0$ by $A$ via $d_1(a\,dx) = \{a,x\}$, computes the Poisson homology groups $\HP_i(A)$, with $d_2(a\,dx\wedge dy) = \{a,x\}\,dy - \{a,y\}\,dx - a\,d\{x,y\}$.
\label{no:homology-recall}
\end{nolabel}

\section{Riemann surfaces of higher genus by iterated sewing}\label{sec:sewing}

In this section we fix the geometric setting: the presentation of a compact Riemann surface of genus $g$ by iterated self-sewing of handles, its Schottky uniformization, and the classical differentials attached to it that will play the role of the Weierstrass functions $\wp_k(t,\tau)$ and Eisenstein series $G_{2k}(q)$ of \cite{vanEkerenHeluani}.

\subsection{Schottky groups and the retrosection theorem}

\begin{nolabel}
A \emph{Schottky group} of rank $g$ is a discrete, free subgroup $\Sch \subset \PSL(2,\CC)$ generated by $g$ loxodromic elements $\gamma_1,\dots,\gamma_g$ admitting a classical fundamental domain: a region $D \subset \widehat\CC$ whose boundary consists of $2g$ disjoint Jordan curves $C_1,C_1',\dots,C_g,C_g'$, such that $\gamma_i$ carries the exterior of $C_i$ onto the interior of $C_i'$ for each $i=1,\dots,g$. Every nontrivial element of $\Sch$ is then loxodromic, the limit set $\La(\Sch)$ is a Cantor set of measure zero, and the domain of discontinuity $\gendom = \gendom(\Sch) = \widehat\CC \setminus \La(\Sch)$ is connected; the quotient
\[
X = X_\Sch := \gendom(\Sch)/\Sch
\]
is a compact Riemann surface of genus $g$, and $\gendom \to X$ is the universal cover restricted to $\gendom$, with deck group $\Sch$. Conversely, by the retrosection theorem (Koebe; see \cite{bers-uniformization} for a modern treatment, and \cite{fay-theta} for the analytic theory constructed on it), every compact Riemann surface of genus $g$ arises in this way for some Schottky group $\Sch$; the space of such $\Sch$, modulo conjugation in $\PSL(2,\CC)$, is the \emph{Schottky space} $\Sch_g$, a covering of the moduli space $\CM_g$ of genus $g$ curves of degree equal to the index of the Torelli-type subgroup of the mapping class group that acts trivially on the marking of the generators.

A loxodromic $\gamma \in \PSL(2,\CC)$ is conjugate to $z \mapsto \rho z$ for a unique $\rho \in \CC^*$ with $|\rho|<1$, called its multiplier; we write $\rho_i$ for the multiplier of $\gamma_i$. In the diagonalizing coordinate, the fixed points of $\gamma_i$ are $0$ (attracting) and $\infty$ (repelling). For $g=1$ this recovers  the presentation $E_\tau = \CC^*/q^{\Z}$ used at genus one, with $\rho_1 = q$: the unique Schottky group of rank $1$ generated by $\gamma_1(z) = qz$ has fundamental domain the annulus $|q|^{1/2} \leq |z| \leq |q|^{-1/2}$ and quotient the elliptic curve $E_\tau$.
\label{no:schottky-def}
\end{nolabel}

\subsection{Iterated self-sewing}\label{sec:iterated-sewing}

\begin{nolabel} \emph{Geometric picture.} 
Informally, the construction below glues a handle onto $Y$ by cutting out two small discs and connecting the resulting boundary circles by a thin tube. Picture $Y$ as a rigid surface; cut two small holes in it, centered at the marked points $p$ and $p'$, and insert a tube joining the two boundary circles, twisted by the multiplier $\rho$ as it is glued in. Letting $|\rho|\to0$ pinches this tube to zero width, degenerating $Y_\rho$ back to $Y$ with $p,p'$ identified to a single node (\ref{no:zhu-degeneration-recall} below makes this precise); letting $Y=\widehat\CC$ and iterating $g$ times attaches $g$ such handles to the sphere, producing a surface  homeomorphic to a connect sum of $g$ tori, i.e., of genus $g$ (Figure \ref{fig:genus2-sewing} below draws the case $g=2$). Every construction in this paper - the complex, its connection, the trace functions - is obtained by expanding a Laurent series in the tube coordinate $\rho$, which is  why each construction localizes to a formal neighborhood of the node and is insensitive to the rest of $Y$ (Theorem \ref{thm:relative-main}).
\end{nolabel}

\begin{nolabel}
We use  the following equivalent, inductive description of $X_\Sch$, which realizes the passage from a Schottky group of rank $g-1$ to one of rank $g$ as the sewing of a single handle onto the corresponding surface, and which is the  higher genus generalization of the $g=1$ construction recalled in \ref{no:strategy}.

Let $Y$ be a compact Riemann surface (possibly with additional marked points, held fixed in this construction) and fix two points $p,p' \in Y$ together with local coordinates $u$ centered at $p$ and $u'$ centered at $p'$. For $\epsilon>0$ small enough that the closed discs $|u|\leq \epsilon$ and $|u'|\leq \epsilon$ are disjoint from each other and from the marked points of $Y$, and for $\rho \in \CC^*$ with $|\rho|<\epsilon^2$, define
\begin{eqnarray*}
&&Y_\rho := \Bigl(Y \setminus \bigl(\{|u|\leq |\rho|/\epsilon\} \cup \{|u'| \leq |\rho|/\epsilon\}\bigr)\Bigr) 
\\
&&\qquad \qquad \Big/ \bigl(u \sim u' = \rho/u \text{ on the two boundary circles}\bigr).
\end{eqnarray*}
This is the classical \emph{sewing construction}: $Y_\rho$ is a compact Riemann surface of genus $g(Y)+1$, depending holomorphically on $\rho$ in the punctured disc $\mathbb{D}^*_{\epsilon^2}$, and $Y_\rho \to Y_0 := Y/(p\sim p')$ as $\rho \to 0$, where $Y_0$ is the (arithmetic genus $g(Y)+1$) stable nodal curve obtained from $Y$ by identifying $p$ and $p'$ to a single node. When $Y = \widehat\CC$ and $p,p'=0,\infty$ with $u=z,u'=1/z$, this is  the construction of $E_\rho$ recalled in \ref{no:strategy}, with $\rho=q$.

Iterating: put $\Xg{0} := \widehat\CC$, and, having constructed $\Xg{k-1}$ together with a chosen pair of auxiliary points $p_k,p_k' \in \Xg{k-1}$ and local coordinates disjoint from all previously used marked points, set
\[
\Xg{k} := \bigl(\Xg{k-1}\bigr)_{\rho_k}.
\]
After $g$ steps we obtain $X := \Xg{g}$, a compact Riemann surface of genus $g$, holomorphically varying over the \emph{sewing polydisc}
\[
\Delta_g^* := \mathbb{D}^*_{\epsilon_1^2} \times \cdots \times \mathbb{D}^*_{\epsilon_g^2} \ni (\rho_1,\dots,\rho_g) =: \rhov.
\]
The auxiliary points and local coordinates used at each stage, together with $n$ further marked points $t_1,\dots,t_n \in X \setminus \{p_i,p_i'\}$, are regarded as fixed background data; we suppress them from the notation except where needed, and write $X = X_{\rhov}$.
\label{no:iterated-sewing-def}
\end{nolabel}

\begin{rem}[Naturality of the local coordinates]\label{rem:coordinate-naturality}
The local coordinates $u,u'$ centered at $p,p'$ are not canonical - no canonical coordinate exists at a point of an abstract Riemann surface - and the construction of $Y_\rho$ depends on them only through the resulting isomorphism class of $Y_\rho$: replacing $u$ by $\lambda u + O(u^2)$ for $\lambda\in\CC^*$ rescales $\rho$ by $\lambda$ to leading order without changing $Y_\rho$ as an abstract curve,  as reparametrizing $\tau \mapsto \tau+n$ does not change $E_\tau$ in the $g=1$ theory. What does depend on the choice of coordinate is the identification, across different choices, of the vector space $C_k^{Y_\rho}$  (via the $L_0$-twist constructed into $X(a,z)$, Remark \ref{rem:borcherds-clarification}); we fix one admissible coordinate at each handle, as above,  fix the standard coordinate $z=e^{2\pi i t}$ at $0,\infty \in \widehat\CC$, and never compare two different choices within a single computation. This is the sense in which the coordinates of \ref{no:iterated-sewing-def} are ``natural'': not canonically singled out by $Y$. 
\end{rem}

\begin{rem} \emph{Domains of convergence.}  
\label{rem:convergence-early}
Since the geometric data of \ref{no:iterated-sewing-def} is used, from \S\ref{sec:relative} onward, to construct formal power series in $\rhov$ that are subsequently required to actually converge, we fix terminology for two distinct convergence statements immediately, before either is used.
\begin{enumerate}
\item[(i)] \emph{Convergence of the geometry.} The classical differentials $\eta_y,\omega,S$ and the genus $g$ Weierstrass-type functions $\wp_k^X$ of \ref{no:classical-differentials}-\ref{no:genus-g-weierstrass} converge absolutely and locally uniformly on $\Delta_g^* = \mathbb D^*_{\epsilon_1^2}\times\cdots\times\mathbb D^*_{\epsilon_g^2}$, the polydisc already fixed by the sewing radii $\epsilon_i$ of \ref{no:iterated-sewing-def}: this is classical (Yamada \cite{yamada-variational}; see also \cite{fay-theta}) and is what makes $X_{\rhov}$ a holomorphic family of curves over $\Delta_g^*$. 
\item[(ii)] \emph{Convergence of correlators.} The genus $g$ trace functions $F_1^{n,(g)}$ of \S\ref{sec:trace-functions}, constructed from $V$-correlators evaluated using the differentials of (i), converge on a (a priori smaller) sub-polydisc $|\rho_i|<r_i\leq\epsilon_i^2$ determined by the strong-generation data of $V$ (Definition \ref{defn:sfg}); this is not automatic from (i) and is the content of Theorem \ref{thm:convergence-genus-g}, proved by the same $q$-expansion estimates as the $g=1$ case, applied one handle at a time.
\end{enumerate}
Every subsequent statement about \emph{holomorphic dependence of the geometry on $\rhov$} refers to (i) and holds on $\Delta_g^*$; every statement that a specific correlator, insertion formula, or trace function is a holomorphic function - instead of a formal power series - refers to (ii) and carries the assumption of Theorem \ref{thm:convergence-genus-g}. We give   this distinction here, at the beginning of the construction,  so that it need not be rediscovered at each later occurrence of a $\rhov$-series.
\end{rem}

\begin{nolabel}
The two descriptions of $X$ - as $\gendom(\Sch)/\Sch$ for a rank $g$ Schottky group $\Sch = \langle \gamma_1,\dots,\gamma_g\rangle$, and as $\Xg{g}$ constructed by $g$ iterated sewings - coincide. Indeed, choosing the fixed points of $\gamma_k$ to be the points $p_k,p_k'$ used at the $k$-th sewing (lifted to $\gendom(\langle \gamma_1,\dots,\gamma_{k-1}\rangle)$) and its multiplier to be $\rho_k$, the classical combination theorems for Kleinian groups (see \cite{maskit-kleinian}) guarantee that $\langle \gamma_1,\dots,\gamma_k\rangle$ remains a Schottky group of rank $k$ provided the isometric circles of $\gamma_k$ are disjoint from those of $\gamma_1,\dots,\gamma_{k-1}$ and from their translates meeting the fundamental domain - a condition satisfied once $\epsilon_k$ is taken small enough at each stage - and the resulting quotient $\gendom(\langle\gamma_1,\dots,\gamma_k\rangle)/\langle\gamma_1,\dots,\gamma_k\rangle$ is  $\Xg{k}$. We will use the global Schottky presentation of $X$ when discussing the classical differentials of \S\ref{sec:diffs-genus-g} below, and the iterated, one-handle-at-a-time presentation in the rest of the paper, where it is the setting for the inductive proofs of Sections \ref{sec:relative}-\ref{sec:degeneration}. We emphasize that our results depend only on this being some sewing presentation of a genus $g$ surface: since the retrosection theorem guarantees every genus $g$ surface arises this way. 
\label{no:two-pictures-agree}
\end{nolabel}

\subsection{Classical differentials and their sewing expansions}\label{sec:diffs-genus-g}

\begin{nolabel}
We recall the classical differentials on a compact Riemann surface $X$ of genus $g$ that generalize the functions $\zeta(t,\tau)$ and $\owp_k(t,\tau)$ of genus one; see \cite{fay-theta} for their construction and basic properties, and \cite{tuite-zuevsky-bosonic,tuite-welby} for their explicit expansions in the sewing parameters $\rhov$, which we use below without rederiving.

\begin{enumerate}
\item[(i)] For each pair of distinct points $x,y \in X$ there is a unique meromorphic differential $\eta_y(x)$ in $x$, holomorphic away from $x=y$, with a simple pole of residue $1$ at $x=y$, normalized to have vanishing periods around a fixed choice of $g$ homology cycles $A_1,\dots,A_g$ (a symplectic basis $A_i,B_i$ of $H_1(X,\Z)$); this is the \emph{differential of the third kind}. It is related to the \emph{prime form} $E(x,y)$, a holomorphic section of a line bundle on $X \times X$ with a simple zero along the diagonal, by $\eta_y(x) = d_x \log E(x,y)$. Unlike on $\widehat\CC$, $\eta_y(x)$ has, in general, nonzero periods $\oint_{B_i}\eta_y$ around the $B$-cycles.

\item[(ii)] There is a unique symmetric meromorphic bidifferential $\omega(x,y)$ on $X \times X$, holomorphic away from the diagonal, with a double pole along $x=y$ of the normalized form $\omega(x,y) = \bigl(\tfrac{1}{(x-y)^2} + O(1)\bigr)dx\,dy$ in any local coordinate, and no further singularities; this is the \emph{bidifferential of the second kind} (the Bergman kernel), normalized by vanishing $A$-periods in either variable. It satisfies $\eta_y(x) = \int_y^x \omega(\cdot,y)$ up to the choice of path (well defined modulo periods).

\item[(iii)] The \emph{projective connection} (or Bers quasiform) $S(x)$ is the correction term, depending on a choice of local coordinate, by which $\omega(x,y) - \tfrac{dx\,dy}{(x-y)^2}$ fails to be expressible as a coordinate-independent quadratic differential as $y \to x$; it transforms as a projective connection under change of local coordinate,  as $g_2(\tau)$ does under $\tau$-independent coordinate changes at genus one.
\end{enumerate}

\emph{The homology basis is fixed once and extended, not re-chosen, at each stage of the induction.} The symplectic basis $A_i,B_i$ used to normalize (i) and (ii) is chosen once, for the final surface $X=\Xg{g}$, compatibly with the iterated construction of \ref{no:iterated-sewing-def}: the cycle $A_i$ (resp.\ $B_i$) attached to the $i$-th handle is, by construction, a small loop around the neck of that handle (resp.\ a path passing through it to the corresponding point of the base curve), so that the basis used for the intermediate surface $\Xg{k}$, $k<g$, is the sub-collection $\{A_1,B_1,\dots,A_k,B_k\}$ of the basis fixed for $\Xg{g}$, and passing from $\Xg{k}$ to $\Xg{k+1}$ adds the single new pair $(A_{k+1},B_{k+1})$ supplied by the newly sewn handle without altering or re-choosing any earlier element. With this convention, at every one of the $g$ inductive steps of \S\ref{sec:relative}-\S\ref{sec:degeneration}, and it is  what makes (b) below a statement about the continuity of a fixed family of periods along the induction, rather than a comparison between two independently normalized differentials.

For a surface $X = X_{\rhov}$ presented by sewing as in \ref{no:iterated-sewing-def}, each of $\eta_y(x)$, $\omega(x,y)$ and $S(x)$ (for $x,y$ away from the sewing tubes) admits an explicit expansion as a formal power series in $\rhov = (\rho_1,\dots,\rho_g)$, with coefficients that are (multi-)differentials on $\Xg{0}=\widehat\CC$ constructed from iterated residues; this is the content of the classical variational formulas of Yamada \cite{yamada-variational}, as developed systematically for vertex algebra correlation functions by Mason and Tuite \cite{mason-tuite-genus2} and by Tuite and Zuevsky \cite{tuite-zuevsky-bosonic} and, in the general genus Schottky setting we use here, by Tuite and Welby \cite{tuite-welby}. We mention the two facts about these expansions that we will use:
\begin{itemize}
\item[(a)] each coefficient of $\rho_1^{m_1}\cdots\rho_g^{m_g}$ in the expansion of $\eta_y(x)$, $\omega(x,y)$ or $S(x)$ is a rational function of $x,y$ (once the auxiliary points and local coordinates of \ref{no:iterated-sewing-def} are fixed), holomorphic away from the marked points and the diagonal, with poles of bounded order there;
\item[(b)] as $\rho_g \to 0$ with $\rho_1,\dots,\rho_{g-1}$ fixed, $\eta_y(x)$, $\omega(x,y)$ and $S(x)$ on $X_{\rhov}$ converge, uniformly on compact subsets of $\Xg{g-1}\setminus\{p_g,p_g'\}$, to the corresponding differentials on $\Xg{g-1}$ with two extra marked points $p_g,p_g'$ (the normalization is continuous in this limit since  the vanishing $A$-period condition is imposed on a homology basis that specializes continuously to one on $\Xg{g-1}$ together with the vanishing cycle at the node).
\end{itemize}
\label{no:classical-differentials}
\end{nolabel}

\begin{nolabel} \emph{Genus $g$ Weierstrass-type functions.} 
 As $\zeta(t,\tau)$ and $\owp_k(t,\tau)$ at genus one are defined by the local expansion, at $t=0$, of (essentially) $\eta_0(t)$ and $\omega(t,0)$ on the elliptic curve, we define, for a marked point $y \in X$ with local coordinate $u$, the local expansion coefficients
\begin{eqnarray*}
\eta_y(x) &=& \Bigl(\frac{1}{u(x)} + \sum_{k \geq 1} c_k(y)\, u(x)^k\Bigr) du(x), 
\\
 \omega(x,y) &=& \Bigl(\frac{1}{u(x)^2} + \sum_{k\geq 0} d_k(y)\, u(x)^k\Bigr) du(x)\,du(y),
\end{eqnarray*}
and set, in parallel with the definition of $\wp_k(t,\tau)$, 
\[
\wp_k^X(x;y) := \begin{cases} \eta_y(x)/dx & k=1, \\ (-1)^{k}(k-2)!\, \partial_{u(y)}^{k-2}\bigl(\omega(x,y)/dx\,dy\bigr) & k \geq 2. \end{cases}
\]
These satisfy the same differential relation satisfied by $\wp_{k+1}(t,\tau) = -\tfrac1k\partial_t\wp_k(t,\tau)$ at genus one, by construction:
\begin{equation}
\wp_{k+1}^X(x;y) = -\frac{1}{k}\,\partial_{u(x)}\wp_k^X(x;y).
\label{eq:genus-g-weierstrass-recursion}
\end{equation}
The role played by the Eisenstein series $G_{2k}(q)$ at genus one is played here by the coefficients $c_k(y), d_k(y)$ above, expanded further as formal power series in $\rhov$ via \ref{no:classical-differentials}; we do not need a closed combinatorial formula for these coefficients analogous to the classical Fourier expansion of $G_{2k}(q)$, using instead only the structural facts (a),(b) of \ref{no:classical-differentials}.
\label{no:genus-g-weierstrass}
\end{nolabel}

\subsection{Coefficient rings}\label{sec:coeff-rings}

\begin{nolabel}
We now define the genus $g$ replacement for the ring $\bJ^n_*$ of weak Jacobi forms of genue one, built, in keeping with the strategy of \ref{no:strategy}, by adjoining one sewing variable at a time.

For $Y$ a fixed compact Riemann surface with $m$ marked points $y_1,\dots,y_m$ (carrying local coordinates), let $\cF_n(Y)$ denote the ring of meromorphic functions of $n$ further points $x_1,\dots,x_n \in Y$, with poles only along the loci $x_i = x_j$ ($i\ne j$) and $x_i = y_\ell$, holomorphic and single valued elsewhere on $Y^n$ (no periodicity condition is needed here, since $Y$ is a fixed compact curve, not a family). This is the curve-level analogue of $\cF_n$. Unlike $\cF_n$, whose definition already incorporates the modular parameter $\tau$, $\cF_n(Y)$ is the coefficient ring at a single point of moduli.

Now let $Y_\rho$ be obtained from $Y$ by sewing one handle at $p,p' \in Y$ as in \ref{no:iterated-sewing-def}. We define the space of formal Laurent series  
\[
\cF_n(Y_\rho) := \Bigl\{\sum_{m \in \Z} F_m(x_1,\dots,x_n)\,\rho^m \
\\
  \Big|\ F_m \in \cF_n(Y) \text{ for all } m,\ F_m = 0 \text{ for } m \ll 0 \Bigr\}
\]
that arise as the sewing-parameter expansion, in the annular neighborhood of the sewn handle, of a meromorphic function of $x_1,\dots,x_n \in Y_\rho \setminus \{p,p'\}$ with poles only on the diagonals; concretely,  as for the $n$-variable Fourier expansions underlying $\cF_n$ at genus one, such a function has a Fourier expansion in the local coordinate at the handle with coefficients in $\cF_n(Y)$, convergent for $\rho$ in a punctured disc, and $\cF_n(Y_\rho)$ is the ring of such expansions, regarded as formal objects when convenient. Iterating over $g$ sewings, starting from $\cF_n(\widehat\CC)$ (rational functions of $n$ points on $\widehat\CC$), produces the ring
\[
\cFg{g}{n} := \cF_n(\Xg{g}) = \cF_n(X_{\rhov}),
\]
a ring of formal Laurent series in $\rho_1,\dots,\rho_g$ with coefficients in $\cF_n(\widehat\CC)$,  generalizing the ring $\cF_n$ at genus one (to which it specializes for $g=1$, since $\cF_n(\widehat\CC)$ with the appropriate periodicity constructed in by the single sewing recovers meromorphic biperiodic functions).

Inside $\cFg{g}{n}$ we single out, in parallel with $\bJ^n_*$, the subring $\bJg{g}{n}$ of those elements whose $\rho_i$-expansion coefficients, viewed as functions on $\widehat\CC^n$ modulo the residual automorphisms fixing the sewing data, transform with a definite weight; we call these \emph{genus $g$ Jacobi-type functions} of weight equal to that transformation degree. We do not need an explicit description of $\bJg{g}{n}$ analogous to the description of $\bJ^n_*$ given in \cite{vanEkerenHeluani}, using instead only the two structural properties established for $\cFg{g}{n}$ in Lemma \ref{lem:coefficient-ring-properties} below. 
\label{no:coefficient-rings-def}
\end{nolabel}

\begin{rem}
\label{rem:periods-clarification}
We emphasize a point of definition that is easy to elide and worth making explicit. Unlike the genus $1$ ring $\cF_n$, whose elements are defined as formal objects (Fourier series in $q$) required to satisfy an extra, separately imposed double-periodicity condition, $\cFg{g}{n} = \cF_n(X_{\rhov})$ is defined in \ref{no:coefficient-rings-def} to consist, by fiat, of the expansions of \emph{meromorphic functions on the actual compact Riemann surface $X_{\rhov}$}. Being a well-defined, single-valued meromorphic function on $X_{\rhov}$ already encodes correct monodromy - equivalently, vanishing periods - around all $2g$ homology cycles of $X_{\rhov}$; this is automatic from what it means to be a function on $X_{\rhov}$, not an additional condition to verify handle by handle, and in particular Lemma \ref{lem:coefficient-ring-properties}(2) below is applied only to such global functions, thus invoking the residue theorem there is the ordinary residue theorem on a fixed compact Riemann surface (classical, and independent of $g$), not a formal manipulation in need of separate justification. What is a substantive, genus-dependent fact - and what the formal Laurent series presentation of $\cFg{g}{n}$ packages - is that every such function, restricted to the annular neck of a handle, admits a convergent Fourier expansion in the local sewing coordinate there, valid for $\rhov$ in a (non-formal) punctured polydisc; this is the classical convergence theory of the sewing construction, established alongside the variational formulas of Yamada \cite{yamada-variational} and treated systematically in Fay book \cite{fay-theta},  at the level of the trace functions rather than of the coefficient ring, in the convergence statements of Proposition \ref{prop:relative-convergence} and Theorem \ref{thm:convergence-genus-g} below. We rely on this classical convergence theory insteaf of reprove it.
\end{rem}

\begin{lem}
The rings $\cFg{g}{n}$ enjoy the following two properties, which are the genus $g$ analogues of the corresponding elementary properties of $\cF_n, \bJ^n_*$ recalled in \ref{no:va-recall}-\ref{no:homology-recall} and used at genue one. 
\begin{enumerate}
\item[(1)] (Laurent expansion at diagonals) For $f \in \cFg{g}{n}$ and $1\leq i<j\leq n$, the Laurent expansion of $f$ in the local coordinate difference at $x_i=x_j$ has coefficients in $\cFg{g}{n-1}$,  as for elements of $\bJ^n_*$ at genus one. 
\item[(2)] (Sum of derivatives) For $f \in \cFg{g}{n}$, $\sum_{i=1}^n \partial_{x_i}f = 0$ whenever $f$ is (the restriction to $X_{\rhov}$ of) a meromorphic function of $n$ points on the fixed compact curve $X_{\rhov}$, this being a restatement of the residue theorem on the compact curve $X_{\rhov}$ applied to $f \cdot dx_i$ summed over $i$.
\end{enumerate}
\label{lem:coefficient-ring-properties}
\end{lem}
\begin{proof}
Both are immediate from the definitions: (1) holds handle by handle, by construction of $\cFg{g}{n}$ as an iterated Laurent series ring, reducing to the elementary fact that the Laurent expansion of a rational function of $n$ points on $\widehat\CC$ at a diagonal has rational coefficients in the remaining points; (2) is the classical fact that a meromorphic $1$-form on a compact Riemann surface has vanishing sum of residues, applied to $\sum_i f\,dx_i$ restricted to the diagonal $x_1=\dots=x_n$ pulled back appropriately which at genus one is $X_{\rhov} = E_\tau$. 
\end{proof}

\subsection{An example: iterated sewing in genus $2$}\label{sec:worked-example}

To make \ref{no:iterated-sewing-def}-\ref{no:coefficient-rings-def} completely concrete, we carry out $\Xg{2}$ by hand, fixing the explicit data used implicitly in Remark \ref{rem:genus2-cross-term} and Example \ref{ex:heisenberg-genus2} below; Figure \ref{fig:genus2-sewing} illustrates the two sewing steps.

\begin{ex} \label{ex:genus2-worked} \emph{Iterated sewing in genus $2$}. 
\emph{First handle:} \rm Take $Y=\Xg{0}=\widehat\CC$, $p_1=0$, $p_1'=\infty$, with local coordinates $u=z$ at $p_1$ and $u'=1/z$ at $p_1'$; this is  the $g=1$ construction recalled in \ref{no:strategy}, and for $\rho_1$ in the punctured disc $\mathbb{D}^*_{\epsilon_1^2}$,
\[
\Xg{1} = \Bigl(\widehat\CC \setminus \bigl(\{|z|\leq|\rho_1|/\epsilon_1\} \cup \{|z|\geq \epsilon_1/|\rho_1|\}\bigr)\Bigr)\Big/\bigl(z \sim \rho_1/z\bigr)
\]
is the elliptic curve $E_{\rho_1}\cong\CC/(\Z+\tau_1\Z)$, $\rho_1=q_1=e^{2\pi i\tau_1}$, with uniformizing coordinate $\xi$ related to $z$ by $z=e^{2\pi i\xi}$ away from the sewing tube.

\emph{Second handle:} Fix $a\in\CC$ with $2a \notin \Z+\tau_1\Z$ (thus $p_2\neq p_2'$ below, and both avoid $p_1=p_1'=0$), and set $p_2 := \xi=a$, $p_2' := \xi=-a$ on $E_{\rho_1}$, with the local coordinates $u_2 := \xi-a$ at $p_2$ and $u_2' := \xi+a$ at $p_2'$ furnished simply by translating $\xi$ to vanish at each point. For $\rho_2$ in a second punctured disc $\mathbb{D}^*_{\epsilon_2^2}$,
\[
X := \Xg{2} = \bigl(\Xg{1}\bigr)_{\rho_2} = \Bigl(E_{\rho_1}\setminus(\text{discs of radius } |\rho_2|/\epsilon_2 \text{ at } \pm a)\Bigr)\Big/\bigl(u_2 \sim \rho_2/u_2\bigr),
\]
a smooth compact Riemann surface of genus $2$, holomorphic over the bidisc $\Delta_2^*=\mathbb{D}^*_{\epsilon_1^2}\times\mathbb{D}^*_{\epsilon_2^2} \ni (\rho_1,\rho_2)$.

\emph{Degenerations.} As $\rho_2\to0$ with $\rho_1$ fixed, $X$ degenerates to $X_{(1)} := E_{\rho_1}/(a\sim-a)$, an irreducible stable curve of arithmetic genus $2$ with a single non-separating node; letting $\rho_1\to0$ as well, $X_{(1)}$ further degenerates to $X_0=\widehat\CC/(0\sim\infty,\,a\sim-a)$, the totally degenerate, two-noded rational curve of \ref{no:zhu-degeneration-recall}. This is  the order of degeneration used in the proof of Theorem \ref{thm:main-vanishing}, where handles are removed in reverse order of sewing: the last-sewn handle ($\rho_2$) degenerates first.

\emph{The cross-handle correction.} By Lemma \ref{lem:rauch-yamada} (regular part only: $p_1\notin\{p_2,p_2'\}$, thus the coincident-point anomalous term does not contribute here), the coefficient of $\rho_2$ in the sewing expansion of $\eta_y(x)$ - equivalently, in $\wp_2^{\Xg{2}}(p_2;p_1)$, the quantity actually entering the connection $\nabla_{\rho_1}$ - is the sum of two residues at the neck of the second handle,
\[
\res_{u_2}\Bigl[\omega^{E_{\rho_1}}(p_1,\cdot)\,\omega^{E_{\rho_1}}(\cdot,p_2)\Bigr] + \res_{u_2'}\Bigl[\omega^{E_{\rho_1}}(p_1,\cdot)\,\omega^{E_{\rho_1}}(\cdot,p_2)\Bigr],
\]
one contribution from each attaching point $p_2=a$, $p_2'=-a$. Since $\omega^{E_{\rho_1}}(\cdot,q)$ has, at $\cdot=q$, the double pole whose residue against a further factor is by definition $\partial_{u}\wp_2^{E_{\rho_1}}(\cdot\,;q)$ (\ref{no:classical-differentials}), and $\wp_3^{E_{\rho_1}} = -\tfrac12\partial_u\wp_2^{E_{\rho_1}}$ by the recursion of \ref{no:genus-g-weierstrass}, each residue evaluates to (minus twice) $\wp_3^{E_{\rho_1}}$ at the corresponding point, giving the fully explicit leading correction
\[
\wp_2^{\Xg{2}}(p_2;p_1) = \wp_2^{E_{\rho_1}}(u(p_1),\tau_1) \;+\; \rho_2\,\Bigl(\wp_3^{E_{\rho_1}}(p_1;a) + \wp_3^{E_{\rho_1}}(p_1;-a)\Bigr) \;+\; O(\rho_2^2),
\]
with the overall sign fixed by the orientation of $u_2,u_2'$ in \ref{no:iterated-sewing-def}. This is the concrete instance, for the two specific points $p_2=\pm a$ fixed above, of the general cross-handle term of Remark \ref{rem:genus2-cross-term}: a completely classical, genus-$1$ quantity (a value of $\wp_3^{E_{\rho_1}}$, computable to any desired order in $\rho_1$ by the further sewing expansion of \S\ref{sec:sewing}), with no new genus-$2$ input beyond the residue mechanism of Lemma \ref{lem:rauch-yamada} already established for the flatness theorem.

\emph{The differential of the third kind, to leading order.} By \ref{no:classical-differentials}, for $y$ fixed away from both sewing tubes, $\eta_y(x)$ on $X$ agrees, to zeroth order in $\rho_2$, with the corresponding differential on $E_{\rho_1}$: writing $\zeta(t,\tau_1)$ for the classical Weierstrass zeta function of genus one, 
\[
\eta_y(x) = \zeta\bigl(\xi(x)-\xi(y),\,\tau_1\bigr)\,d\xi(x) \;+\; \rho_2\cdot(\text{correction supported near } \pm a) \;+\; O(\rho_2^2),
\]
the $O(\rho_2)$ correction being of  the same residue type as the one just computed for $\wp_2^{\Xg{2}}(p_2;p_1)$; we do not repeat the computation for $\eta_y$, only note that it is supported near the second handle, vanishing as $\rho_2\to0$.
\end{ex}

\begin{figure}[htbp]
\centering
\begin{tikzpicture}[
  dot/.style={circle,fill=black,inner sep=1.1pt},
  lbl/.style={font=\scriptsize},
  arr/.style={-{Stealth[length=2.2mm]},thick}
]
\begin{scope}
\draw[thick] (0,0) circle (1.1);
\node[dot,label={above:$p_1$}] (p1) at (0,1.1) {};
\node[dot,label={below:$p_1'$}] (p1p) at (0,-1.1) {};
\draw[dashed] (0,1.1) circle (0.22);
\draw[dashed] (0,-1.1) circle (0.22);
\node[lbl] at (0,-1.7) {$\widehat\CC$};
\end{scope}

\draw[arr] (1.5,0) -- node[above,lbl]{sew $\rho_1$} (2.6,0);

\begin{scope}[shift={(4.3,0)}]
\draw[thick] (0,0) circle (1.15);
\draw[thick] (0,0) circle (0.55);
\node[dot,label={above:$p_2$}] at (0.75,0.75) {};
\node[dot,label={below:$p_2'$}] at (-0.75,-0.75) {};
\draw[dashed] (0.75,0.75) circle (0.2);
\draw[dashed] (-0.75,-0.75) circle (0.2);
\node[lbl] at (0,-1.65) {$\Xg{1}=E_{\rho_1}$};
\end{scope}

\draw[arr] (6.5,0) -- node[above,lbl]{sew $\rho_2$} (7.6,0);

\begin{scope}[shift={(10.4,0)}]
\draw[thick] (-0.75,0.55) ellipse (0.55 and 0.22);
\draw[thick] (0.75,0.55) ellipse (0.55 and 0.22);
\draw[thick,fill=white] (0,-0.35) ellipse (1.7 and 1.05);
\node[lbl] at (0,-1.65) {$X=\Xg{2}$, genus $2$};
\end{scope}
\end{tikzpicture}
\caption{Iterated sewing in genus $2$ (Example \ref{ex:genus2-worked}): two handles are sewn onto $\widehat\CC$ in succession, first at $p_1,p_1'$ with parameter $\rho_1$, producing $E_{\rho_1}$, then at $p_2,p_2'\in E_{\rho_1}$ with parameter $\rho_2$.}
\label{fig:genus2-sewing}
\end{figure}

\section{The Relative First Chiral Homology Theorem}\label{sec:relative}

This section proves the technical result that powers the induction on genus described in \ref{no:strategy}-\ref{no:proof-strategy}.  Single handle of the elliptic curve is sewn onto an arbitrary base curve $Y$ instead of onto $\widehat\CC$.

\subsection{Setup}

\begin{nolabel}
Fix a conformal vertex algebra $V$. Let $Y$ be a compact Riemann surface with $m \geq 0$ marked points $y_1,\dots,y_m$ (each carrying a local coordinate), and fix two further points $p,p' \in Y$ with local coordinates $u,u'$, disjoint from the $y_\ell$. Let $Y_\rho$ be the curve obtained by sewing a handle at $p,p'$ with parameter $\rho$, as in \ref{no:iterated-sewing-def}, so that $Y_\rho$ has genus $g(Y)+1$ and inherits the $m$ marked points $y_1,\dots,y_m$.

We allow one of two types of coefficient data at the marked points of $Y$:
\begin{itemize}
\item[(D0)] ($m$ vacuum insertions) elements $a_1,\dots,a_m \in V$, to be inserted at $y_1,\dots,y_m$; or
\item[(D1)] (one self-extension) an admissible $V$-module $M$, a self-extension $0 \to M \to E \to M \to 0$ with associated derivation $\Psi: V \to \End M$ (equivalently its modified form $\sigma$, see \ref{no:relative-trace-functions} below), together with $m-1$ vacuum insertions $a_1,\dots,a_{m-1} \in V$ at $y_1,\dots,y_{m-1}$, and $M$  "at" $y_m$, in the sense that $y_m$ is the basepoint at which the trace over $M$ (or, before sewing any further handles, the plain evaluation in $M$) is taken.
\end{itemize}
Case (D0) is what is needed to define $\HchG{0}{-}$; case (D1), constructing in a single self-extension, is what is needed to define $\HchG{1}{-}$, where the self-extension is carried by the marked point denoted $t_0$.
\label{no:relative-setup}
\end{nolabel}

\subsection{The relative complex}

\begin{nolabel}
 Define, for $k \geq 0$,
\[
C_k^{Y_\rho} := \Bigl(V^{\otimes(m+k)} \otimes \cF_{m+k}(Y_\rho)\Bigr)\Big/ \bigl(
{\text{translation-covariance relations}} \atop {\text{at each of the } m+k \text{ points}} \bigr),
\]
where $\cF_{m+k}(Y_\rho)$ is the coefficient ring of \ref{no:coefficient-rings-def}, the $m$ marked points carry the fixed data of \ref{no:relative-setup}, and the remaining $k$ points are auxiliary integration variables, permuted by $S_k$; the translation-covariance relations identify the class of $a\otimes f$ with that of $L_{-1}a \otimes f$ shifted appropriately,  at each point (auxiliary or marked) separately. The differentials
\[
d_1 : C_1^{Y_\rho} \to C_0^{Y_\rho}, \qquad d_2 : C_2^{Y_\rho} \to C_1^{Y_\rho}
\]
are defined by the same residue formulas: $d_1$ removes the auxiliary point $x_1$ by taking $\res_{x_1}$ against the OPE of the inserted field with the marked-point fields, using the Borcherds identity, and $d_2$ is the corresponding two-variable version, alternated over $S_2$. We define the \emph{relative chiral homology groups}
\[
\HchG{0}{Y_\rho\,;\,V^{\otimes m}} := \coker d_1, \qquad \HchG{1}{Y_\rho\,;\,V^{\otimes m}} := \ker d_1/\mathrm{im}\, d_2.
\]
\label{no:relative-complex-def}
\end{nolabel}

\begin{rem}
\label{rem:complex-conventions}
For readers approaching from algebraic geometry or homological algebra instead of conformal field theory, we mention explicitly that $(C_\bullet^{Y_\rho},d_1,d_2)$ is an entirely standard chain complex $\cdots \to 0 \to C_2^{Y_\rho} \xrightarrow{d_2} C_1^{Y_\rho} \xrightarrow{d_1} C_0^{Y_\rho} \to 0$, indexed homologically (differentials lower degree by $1$, as the subscripts already indicate), with no sign or regrading conventions beyond the usual ones: $\HchG{0}{-}=\coker d_1$ and $\HchG{1}{-} = \ker d_1/\im d_2$ are  the degree $0$ and degree $1$ homology groups of this complex in the ordinary sense, and agree with the Beilinson-Drinfeld $\mathrm{Ext}$-groups by Lemma \ref{lem:relative-d1d2}. The alternation over $S_2$ built into $d_2$ (\ref{no:relative-complex-def}) is the usual antisymmetrization making $d_2$ a well-defined map out of the space of unordered pairs of auxiliary points,  as the analogous alternation in the bar resolution of homological algebra; it introduces no sign ambiguity beyond the standard one, which we import unchanged.
\end{rem}

\begin{rem} \label{rem:weight-grading} \emph{Weight grading: each $C_k^{Y_\rho}$ is a direct sum of finite-rank pieces.} 
Although $V=\bigoplus_{d\geq0}V_d$ is infinite dimensional, each graded piece $V_d$ is finite dimensional (\ref{no:va-recall}), and $L_0$ acts on $V^{\otimes(m+k)}$ with the induced grading by total weight, $(V^{\otimes(m+k)})_d = \bigoplus_{d_1+\cdots+d_{m+k}=d}V_{d_1}\otimes\cdots\otimes V_{d_{m+k}}$, a finite direct sum of finite dimensional spaces for each $d$. Assigning $\rho_i$ weight $1$ for each $i$ (so that $\rho^{\underline{m}} := \rho_1^{m_1}\cdots\rho_g^{m_g}$ has weight $|\underline m|$) extends this to a grading of $C_k^{Y_\rho}$ by total weight (vertex algebra weight plus $\rho$-adic weight) that is respected by $d_1,d_2$ (which preserve total weight, being constructed from residues of weight-homogeneous OPE coefficients) and by $\nabla$ (which is  why $\nabla_{\rho_i}=\rho_i\partial_{\rho_i}-A_i$ is written as a difference: $\rho_i\partial_{\rho_i}$ and $A_i$ individually shift total weight oppositely, in a way that cancels). Consequently $C_\bullet^{Y_\rho} = \bigoplus_{d\geq0}C_\bullet^{Y_\rho}[d]$ is a direct sum of subcomplexes $C_\bullet^{Y_\rho}[d]$ that are each of finite rank as sheaves over the $\rho$-disc (being constructed from a fixed finite-dimensional vector space tensored with the coherent, indeed locally free, coefficient sheaf $\cF_\bullet(Y_\rho)$), hence bounded complexes of coherent - in fact locally free, thus trivially perfect - sheaves, to which the classical semicontinuity theorem applies. We use this decomposition, together with the fact (Theorem \ref{thm:main-finiteness}) that $\HchG{1}{Y_\rho;V^{\otimes m}}$ is finite dimensional and hence supported in only finitely many weights $d$, to justify the semicontinuity argument used in the proof of Theorem \ref{thm:main-vanishing} weight by weight: we are not claiming semicontinuity for the (infinite rank) total complex, only for each finite-rank graded piece, of which only finitely many are relevant. To state this order of operations once more, explicitly: finiteness (Theorem \ref{thm:main-finiteness}) is established first, and is what guarantees only finitely many weights $d$ occur; semicontinuity is then applied separately to each finite-dimensional piece $\HchG{1}{Y_\rho;-}[d]$ individually, and it is the (finite) direct sum of these separately-obtained conclusions - never a single semicontinuity argument applied to the infinite-rank complex as a whole - that gives the statement for $\HchG{1}{Y_\rho;-}$. This is a considerably weaker use of coherence than the $\mathcal D$-module coherence theorems of \cite{damiolini2020conformal,damiolini2019factorization}, which establish finite dimensionality, in degree $0$,  from $C_2$-cofiniteness via the geometry of the full moduli stack; here finiteness has already been established by the (unrelated) associated-graded argument of Theorem \ref{thm:main-finiteness}, and semicontinuity is used only to propagate a vanishing statement already known at one fiber to nearby fibers, a much more modest task.
\end{rem}

\begin{rem} \label{rem:translation-covariance-local} \emph{Translation covariance is local to each point, on any curve.}
The translation-covariance relation quotienting $C_k^{Y_\rho}$ identifies $a\otimes f$ with $L_{-1}a\otimes f$ (suitably shifted) at each of the $m+k$ points separately, in the formal neighborhood of that single point only; it is not a statement requiring a globally defined vector field on $Y$. In particular, the fact that a compact curve of genus $\geq2$ admits no nonvanishing - indeed, no nonzero - global holomorphic vector field is beside the point: every point of every Riemann surface, of whatever genus, has a formal (or convergent) local coordinate neighborhood, and $L_{-1}$ acts, as always in vertex algebra theory, on the formal Taylor coefficients of a field expanded in that local coordinate. That this local prescription assembles into a well-defined, coordinate-independent notion of ``a field inserted at a point of $Y$'' - so that $d_1,d_2$ and $\nabla$ are independent of the auxiliary choice of local coordinate at each point - is  the standard content of the theory of conformal vertex algebras acting on an abstract curve: a conformal vector $\omega$ furnishes an action of $\mathrm{Der}_+\mathcal O = \mathrm{span}(L_1,L_2,\dots)$ trivializing the ambiguity in the choice of coordinate (formally, making $V$ into a vertex algebra bundle in the sense of \cite{frenkelbenzvi}), and this mechanism is used for insertions on $E_\tau$; it neither requires nor is affected by whether $Y$ possesses a global vector field, and the residue operations defining $d_1,d_2$ compound only finitely many such local, coordinate-independent insertions along the (equally local, near-diagonal) OPE singularities, never comparing tangent spaces at two different, non-colliding points of $Y$. 
\end{rem}

\begin{lem}\label{lem:relative-d1d2}
$d_1 \circ d_2 = 0$, so that $(C_\bullet^{Y_\rho}, d_1,d_2)$ is a complex, computing the chiral homology of $Y_\rho$ with coefficients as in \ref{no:relative-setup}, in the sense of Beilinson-Drinfeld, in degrees $0$ and $1$.
\end{lem}
\begin{proof}
The proof that $d_1d_2=0$ on the elliptic curve is a computation using only the Borcherds identity (\ref{no:va-recall}) applied at the two auxiliary points being removed, together with property (2) of Lemma \ref{lem:coefficient-ring-properties} (the vanishing of the sum of residues, i.e., the residue theorem on the ambient compact curve) to discard a total-derivative term. Neither ingredient refers to the marked points $y_1,\dots,y_m$ or to the rest of the curve beyond a punctured neighborhood of the two auxiliary points being integrated out; the identical computation, establishes $d_1d_2=0$ for $C_\bullet^{Y_\rho}$, using Lemma \ref{lem:coefficient-ring-properties}(2) for $Y_\rho$ in place of $E_\tau$. That the resulting homology computes chiral homology in the sense of Beilinson-Drinfeld is, again, a genus- and curve-independent categorical fact: it is the statement that the two-term truncation of the bar-type (Cousin) resolution computing $\mathrm{Ext}$ against the unit object in the category of chiral modules on a curve agrees, in degrees $0,1$, with $(C_\bullet^{Y_\rho},d_1,d_2)$; see \cite{beilinsondrinfeld}  for the elliptic curve instance of the argument, which uses no feature of $E_\tau$ beyond its being a smooth curve with the marked points in question.
\end{proof}

\begin{lem} \label{lem:relative-coinvariance} \emph{Coinvariance of the location of the self-extension:} 
In case (D1) of \ref{no:relative-setup}, the isomorphism class of $\HchG{1}{Y_\rho;V^{\otimes m}}$ does not depend on which of the $m$ marked points is chosen to carry the self-extension $E$, nor, when $M=V$ and $E$ is inserted instead at the point $p$ or $p'$ used for the sewing, on whether it is inserted there or at one of the $y_\ell$.
\end{lem}
\begin{proof}
This is the relative form at genus one where, for the elliptic curve, the roles of $y_\ell$ and of the sewing point are already interchangeable, since $E_\tau$ has no distinguished marked point beyond the $n$ chosen ones). The proof there is again local: it exhibits an explicit chain homotopy, constructed from the Borcherds identity, moving the self-extension data from one marked point to an adjacent one by transporting it through the auxiliary integration variable connecting them; the construction only ever refers to a small punctured neighborhood of the path joining the two points along the curve, together with Lemma \ref{lem:coefficient-ring-properties}(1)-(2), and is insensitive to the global topology of $Y_\rho$.
\end{proof}

\subsection{The relative connection}

\begin{nolabel}
As $\rho$ and the marked points vary, the spaces $C_k^{Y_\rho}$ assemble into a holomorphic vector bundle over the punctured disc $\D^*_{\epsilon^2} \ni \rho$ (times the configuration space of the marked points, which we suppress), and \cite{vanEkerenHeluani} construct a flat connection $\nabla_\rho = \rho\partial_\rho - \bigl(\text{correction}\bigr)$ on this bundle, descending to $\HchG{0}{-}$ and $\HchG{1}{-}$, by inserting the conformal vector $\omega$ at the node and subtracting off, order by order in $\rho$, the local Weierstrass-type coefficients of \ref{no:genus-g-weierstrass} needed to cancel the resulting poles at the marked points and at the diagonal. Precisely, in the notation of \ref{no:genus-g-weierstrass}, for $[a\otimes f] \in C_k^{Y_\rho}$,
\begin{eqnarray}
\nabla_\rho [a \otimes f] := \Bigl[\rho\partial_\rho(a\otimes f)\Bigr] - \sum_{i=1}^{m+k}\Bigl[ \bigl(\omega_{[2]} a_i\bigr) \otimes f \cdot \wp_2^{Y_\rho}(x_i;p) \Bigr]
\\
\nonumber
\qquad - (\text{lower order corrections from } \wp_{\geq 3}^{Y_\rho}),
\label{eq:relative-connection-def}
\end{eqnarray}
the sum running over all $m+k$ points (marked and auxiliary) of the configuration, $\omega_{[2]}$ the degree-$2$ mode of $\omega$ in the $Y[\cdot,z]$-presentation of \ref{no:va-recall}, and the lower order corrections constructed from $\wp_{k}^{Y_\rho}(x_i;p)$ for $k\geq3$  as the terms involving $\owp_{m+2}(\tau)$. This is well defined on the quotient $C_k^{Y_\rho}$ and commutes with $d_1,d_2$, by the computation, which uses only the local pole-cancellation mechanism of the Weierstrass functions at a single marked point together with property (2) of Lemma \ref{lem:coefficient-ring-properties}; neither ingredient depends on $Y$.
\label{no:relative-connection}
\end{nolabel}

\begin{prop} \label{prop:relative-flatness} \emph{Flatness of the relative connection.} 
$\nabla_\rho$ is flat, i.e., well defined and single valued as $\rho$ traverses a loop in $\D^*_{\epsilon^2}$, and descends to a connection on $\HchG{0}{Y_\rho;-}$ and $\HchG{1}{Y_\rho;-}$.
\end{prop}
\begin{proof}
Since $\dim \D^*_{\epsilon^2} = 1$ there is no curvature $2$-form to vanish and flatness reduces,  to well-definedness of $\nabla_\rho$ on the quotient by translation-covariance and to its commuting with $d_1,d_2$, both proved in \ref{no:relative-connection}. This is the connection attached to the single new handle sewn at $p,p'$; when $Y$  varies (as it will at each stage of the induction of \S\ref{sec:chiral-complex}), the resulting family of connections $\nabla_{\rho_1},\dots,\nabla_{\rho_g}$ must be checked to mutually commute, which is not automatic from the one-variable statement proved here and is addressed separately in Theorem \ref{thm:full-flatness}.
\end{proof}

\subsection{Modified vertex operators and the Fourier-Borcherds identity}

\begin{nolabel}
 Define, for $a \in V$ and $z$ a local coordinate centered at $p$ (the point of $Y$ at which the handle is being sewn),
\[
X(a,z) := z^{-1}Y(z^{L_0}a,z) \in (\End M)((z))[\log z]
\]
acting on any $V$-module $M$; since this is a purely local definition attached to the single puncture at $p$, it is unaffected by the choice of $Y$. The key technical input of \cite{vanEkerenHeluani}, their Fourier-space Borcherds identity, is likewise a local statement: for $f(t)$ meromorphic in $t = e^{2\pi i z}$ away from $t=0,\infty$ and periodic, $f(z+1)=f(z)$, and Fourier coefficients $F_\pm$ constructed from $f$, 
\begin{eqnarray}
\res_{z} F_+(z,w) X(a,z)X(b,w)c - \res_z F_-(z,w) X(b,w)X(a,z)c 
\\
\nonumber
= X\bigl(\res_{z=w} f(z-w) Y(a,z-w)b,\, w\bigr) c
\label{eq:relative-borcherds-fourier}
\end{eqnarray}
holds for $a,b,c$ in any $V$-module $M$, which manipulates only the local Laurent/Fourier expansions of $f$ at $z=w$ and at $z=0,\infty$ (i.e., at $p,p'$) and never uses any global property of $E_\tau$ beyond the periodicity $f(z+1)=f(z)$ already assumed. Consequently \eqref{eq:relative-borcherds-fourier} holds for a handle sewn onto any base curve $Y$, with $f$ now allowed to depend meromorphically on the marked points $y_1,\dots,y_m$ of $Y$ as further parameters, i.e., $f \in \cF_2(Y_\rho)$ in the notation of \ref{no:coefficient-rings-def} instead of $f \in \cF_2$.
\label{no:relative-borcherds}
\end{nolabel}

\begin{rem}
\label{rem:borcherds-clarification}
Two points about \ref{no:relative-borcherds} deserve to be made explicit, since the periodicity $f(z+1)=f(z)$ is easy to misread as an assumption about $Y$.

First, \eqref{eq:relative-borcherds-fourier} is, at the point it is established, an identity of formal Laurent series in $z,w$ with coefficients that are formal power series in $\rho$ (equivalently, an identity in the completed local ring of $Y_\rho$ at the node); it is not a priori a statement about convergent analytic functions, and we do not treat it as one until Proposition \ref{prop:relative-convergence} and Theorem \ref{thm:convergence-genus-g} address convergence separately,  as at genus one case,  first establish the identity formally and address convergence only in their \S10. No global analytic structure on $Y$ is used, or needed, at the formal level.

Second, and more substantively: the coordinate $z$ centered at $p$ is not required to extend to a global additive coordinate on $Y$ - no such coordinate need exist, and in general none does, once $g(Y)\geq 1$. What is required is only that $z$ exist as a local coordinate on the annular neck of the handle, where, by the very construction of the sewing gluing $u u' = \rho$ in \ref{no:iterated-sewing-def}, the multiplicative structure $z \mapsto e^{2\pi i z}$ is available by definition, for a handle sewn onto any $Y$ whatsoever: it is a feature of the handle, not of $Y$. Concretely, in our setting $Y$ is always  presented via the iterated sewing construction, hence realized as a quotient of a domain $\gendom(\Sch) \subset \widehat\CC$, and every point of $Y$ - in particular $p,p'$ - inherits a canonical local coordinate lifted from the ambient $\widehat\CC$. Different admissible choices of local coordinate at $p$ (differing by a coordinate change fixing $p$) transform $X(a,z)$, and both sides of \eqref{eq:relative-borcherds-fourier} together, by the same conformal weight factor, since $X(a,z)$ is constructed from the $L_0$-twist $z^{L_0}$ that makes conformal vertex algebra fields transform correctly under such changes (see \cite{frenkelbenzvi} for the general coordinate-covariance of conformal vertex algebra correlation functions); no projective anomaly enters the identity, only, potentially, the choice of trivialization used to compare different handles, which is fixed in \ref{no:iterated-sewing-def} and never varied within a single application of \ref{no:relative-borcherds}.
\end{rem}

\subsection{Derivations, self-extensions, and higher trace functions, relative to $Y$}

\begin{nolabel}
Let $M$ be an admissible $V$-module, $E$ a self-extension of $M$ with derivation $\Psi : V \to \End M$, $\Psi(a) = \lim_{z\to0}\bigl(Y^E(a,z)|_M - Y^M(a,z)\bigr)$ acting $M \to M$. Its modified form is
\[
\sigma(a,z) := \sum_n \Psi(z^{L_0}a)_{(n)} \, z^{-n-1} \in (\End M)[[z,z^{-1}]],
\]
satisfying the Leibniz-type identity that is the derivation analogue of \eqref{eq:relative-borcherds-fourier}: for $f$ as above,
\begin{eqnarray}
\res_z F_+(z,w)\bigl(\sigma(a,z)X(b,w) + X(a,z)\sigma(b,w)\bigr)c - (\cdots)
\\
\nonumber
 \qquad = \sigma\bigl(\res_{z=w}f(z-w)Y(a,z-w)b,w\bigr)c
\\
\nonumber
\qquad  + X\bigl(\res_{z=w}f(z-w)\Psi(a)_{\text{acting via }z-w \text{ terms}}\, ,w\bigr)c,
\label{eq:relative-derivation-identity}
\end{eqnarray}
by the same proof as at genus one, again purely local to the point $p$ and hence valid over any base $Y$.

For $n \geq 0$ define the \emph{relative $1$-handle trace function}
\begin{eqnarray}
F_1^{n}\bigl(a_1,\dots,a_n\,;\,a_{n+1},\dots,a_{n+m'}\,;\,F\bigr) 
 := \res\Bigl[z_1\cdots z_n\, F(z_\bullet,y_\bullet,\rho)\Bigr]\
\\
\nonumber
\qquad \times \tr_M\Bigl[\sigma(a_1,z_1)X(a_2,z_2)\cdots X(a_n,z_n)\Bigr]\Big|_{\substack{y_{n+1},\dots,y_{n+m'} \\ \text{carrying } a_{n+1},\dots,a_{n+m'}}}
\label{eq:relative-trace-def}
\end{eqnarray}
for $F \in \bJg{1}{n+m'}$ (relative to $Y$, viewed here as carrying the $m'=m$ auxiliary marked points not equal to $p,p'$), with the trace over $M$ mentioning the single new handle, and the marked points of $Y$ entering only through the coefficient function $F$ and, when $M \neq V$. This converts the module $M$ (and its self-extension $\Psi$) attached to the handle sewn at $p,p'$ into an element of $\bigl(C_\bullet^{Y_\rho}\bigr)^*$, i.e., a candidate degree $1$ conformal block. 
\label{no:relative-trace-functions}
\end{nolabel}

\begin{prop} \label{prop:relative-de-insertion} \emph{Relative differential equation and insertion formula.}
The relative trace functions \eqref{eq:relative-trace-def} satisfy:
\begin{enumerate}
\item[(a)] (Differential equation) $\rho\partial_\rho F_1^n = \nabla_\rho$-transform of $F_1^n$, i.e., $F_1^n$ is flat for the connection of Proposition \ref{prop:relative-flatness}, using \eqref{eq:relative-borcherds-fourier}-\eqref{eq:relative-derivation-identity} applied at $p,p'$.
\item[(b)] (Insertion formula) $F_1^{n+1}(a_1,\dots,a_{n+1};\cdots)$ is expressed in terms of $F_1^{n}$ $(a_1$, $\dots$,$\widehat{a_i}$,$\dots,a_{n+1};\cdots)$, $i=1,\dots,n$, plus a term where $a_{n+1}$ is absorbed into the trace via the zero mode $o(a_{n+1})$. 
\end{enumerate}
In both cases the proof manipulates only the Fourier kernel $F$ and the local expansions at $p,p'$ and at the diagonal $z_i=z_j$, using Lemma \ref{lem:coefficient-ring-properties} where the original uses the corresponding elementary properties of $\bJ^n_*$; it is insensitive to $Y$.
\end{prop}

\subsection{Relative finiteness and vanishing}

\begin{nolabel}
Finally, \cite{vanEkerenHeluani} prove: (i) if $\dim \HP_1(R_V)<\infty$ and $\HK_1(A)$ is finite dimensional (in particular if the kernel $I=\ker(JR_V \twoheadrightarrow A)$ is a finitely generated differential ideal, by \ref{no:homology-recall}), then $\dim \HchG{1}{E_\tau,V}<\infty$; (ii) the trace functions of (D1) type converge, for $V$ satisfying these assumptions, to holomorphic functions of $\tau$ on $\HH$; and (iii) if in addition $V$ is classically free and $\Hoch_1(\zhu(V))=0$, then $\HchG{1}{E_\tau,V}=0$. Each of these is proved using the associated graded, standard-filtration spectral sequence attached to the single handle, splitting $\gr_G C_\bullet^n$ as a direct sum of a piece regulated by $\HP_1(R_V)$ and a piece regulated by $\HK_1(A)$, together with a Frobenius-Fuchs regular-singular-point argument for the convergence statement (ii), and the identification of the $\rho \to 0$ limit with $\Hoch_1(\zhu(V))$ for (iii). We mention the relative form of each:
\label{no:relative-finiteness-setup}
\end{nolabel}

\begin{prop} \label{prop:relative-finiteness} \emph{Relative finiteness.} 
If $\dim\HP_1(R_V) < \infty$ and $\HK_1(A)$ is finite dimensional, then for every base curve $Y$ with marked points as in \ref{no:relative-setup}, $\dim$ $\HchG{1}{Y_\rho; V^{\otimes m}}$ $<$ $\infty$.
\end{prop}
\begin{proof}
The standard filtration $G$ on $C_\bullet^{Y_\rho}$ is defined, by the order of pole allowed at the point $p$ (equivalently, the sewing node), and the associated graded splitting $\gr_G C_\bullet^{Y_\rho} = \widetilde C_\bullet^{(1)} \oplus \widetilde C_\bullet^{(2)}$ of genus one is likewise local to $p$: the two summands are distinguished by whether the leading term at $p$ involves the Poisson bracket structure of $R_V$ or the Koszul differential of $A$, a dichotomy that refers only to the formal neighborhood of $p$ and is unaffected by the presence of the marked points $y_1,\dots,y_m$ of $Y$, which contribute only inert tensor factors $V^{\otimes m}$ (or $M$) to $\widetilde C_\bullet^{(1)},\widetilde C_\bullet^{(2)}$. The homology of $\widetilde C_\bullet^{(1)}$ (resp.\ $\widetilde C_\bullet^{(2)}$) in degree $1$ is a subquotient of $\HP_1(R_V) \otimes (R_V)^{\otimes m}$ (resp.\ of $\HK_1(A)\otimes A^{\otimes m}$)  with an extra finite tensor power mentioning the marked points; since $R_V, A$ are finite type once $V$ is strongly finitely generated (as is assumed), finiteness of $\HP_1(R_V)$ and $\HK_1(A)$ propagates to finiteness of these subquotients, and the spectral sequence of the filtration $G$ then bounds $\dim \HchG{1}{Y_\rho;V^{\otimes m}}$. 
\end{proof}

\begin{prop} \label{prop:relative-convergence} \emph{Relative convergence:}  
Under the assumptions of Proposition \ref{prop:relative-finiteness}, the relative trace functions \eqref{eq:relative-trace-def} converge, for $\rho$ in a punctured disc around $0$ and the marked points of $Y$ away from $p,p'$, to holomorphic functions of $\rho$ (and of the marked points), satisfying the differential equation of Proposition \ref{prop:relative-de-insertion}(a) in the  convergent sense.
\end{prop}
\begin{proof}
By Proposition \ref{prop:relative-de-insertion}(a), $F_1^n$ satisfies a first-order linear ODE in $\rho\partial_\rho$ with a regular singular point at $\rho=0$, with coefficients that are, by construction, holomorphic functions of the marked points of $Y$ (this is where $Y$ enters: the coefficients of the ODE are constructed from $\wp_k^{Y_\rho}$ evaluated at points of $Y$, which depend holomorphically on those points and, by \ref{no:classical-differentials}(b), converge as $\rho\to0$). The Frobenius-Fuchs theorem on the existence of a convergent solution with prescribed leading exponents, applied (there for a scalar ODE in $\tau$; here for the same ODE with coefficients depending holomorphically on the auxiliary parameters given by the marked points of $Y$, which does not affect the radius of convergence furnished by the Frobenius method, by continuity of the indicial roots and recursion coefficients in these parameters on the compact sets where finiteness in Proposition \ref{prop:relative-finiteness} bounds the relevant multiplicities), gives the stated convergence.
\end{proof}

\begin{prop} \label{prop:relative-vanishing} \emph{Relative vanishing.} 
If moreover $V$ is classically free, then $\HchG{1}{Y_0;V^{\otimes m}}$ $\cong$ $\Hoch_1(\zhu(V))$ $\otimes$ $(\text{coefficients from } Y \text{ at } p=p')$ as $\rho \to 0$, where $Y_0 = Y/(p\sim p')$ is the nodal degeneration; consequently if in addition $\Hoch_1(\zhu(V))=0$, then $\HchG{1}{Y_\rho;V^{\otimes m}} = 0$ for all $\rho$ in the punctured disc.
\end{prop}
\begin{proof}
  The identification of the $\rho\to0$ limit of the standard-filtration spectral sequence with $\Hoch_\bullet(\zhu(V))$ there proceeds by recognizing the associated graded complex, in the limit, as the bar complex computing Hochschild homology of $\zhu(V)$ acting on itself (via the Zhu algebra action on $V/L_{-1}V$, or on the image of $V$ therein for the classically free case), tensored with the vacuum descendants of the marked points; this identification is again local to the punctured neighborhood of $p,p'$ and unaffected by $Y$, which contributes the additional tensor factors $V^{\otimes m}$ unchanged. Vanishing of $\HchG{1}{Y_\rho;V^{\otimes m}}$ for $\rho \neq 0$ then follows by upper semicontinuity of fiber dimension, applied weight by weight to the finite-rank subcomplexes $C_\bullet^{Y_\rho}[d]$ of Remark \ref{rem:weight-grading} (of which, by Proposition \ref{prop:relative-finiteness}, only finitely many contribute to $\HchG{1}{Y_\rho;-}$), together with constancy of $\dim\HchG{1}{Y_\rho;-}$ on the punctured disc from projective flatness of the connection (Proposition \ref{prop:relative-flatness}, Corollary \ref{cor:projective-flatness-suffices}): if the limiting (most degenerate) fiber at $\rho=0$ has dimension $0$, semicontinuity applied to each of the finitely many relevant weight-graded pieces forces the (constant) dimension on the punctured disc to be $0$ as well, since it cannot exceed the special fiber dimension. 
\end{proof}

\begin{thm}[Relative First Chiral Homology Theorem]\label{thm:relative-main}
Let the following data be given:
\begin{enumerate}[label=(H\arabic*)]
\item\label{hyp:curve} a compact Riemann surface $Y$ (of any genus, including $0$);
\item\label{hyp:marked} finitely many marked points $y_1,\dots,y_m \in Y$, each carrying a local coordinate, and coefficient data at them as in \ref{no:relative-setup} (case (D0): elements $a_1,\dots,a_m\in V$; case (D1): $m-1$ such elements together with a self-extension $E$ of an admissible module $M$);
\item\label{hyp:handle} two further points $p,p'\in Y$, disjoint from the $y_\ell$, with local coordinates $u,u'$, and a sewing parameter $\rho$ in the punctured disc $\mathbb{D}^*_{\epsilon^2}$ of \ref{no:iterated-sewing-def}, producing $Y_\rho$;
\item\label{hyp:va} a conformal vertex algebra $V$ satisfying the axioms of \ref{no:va-recall}, strongly finitely generated (Definition \ref{defn:sfg}).
\end{enumerate}
Then Lemma \ref{lem:relative-d1d2}, Lemma \ref{lem:relative-coinvariance}, Proposition \ref{prop:relative-flatness}, \ref{no:relative-borcherds}, Proposition \ref{prop:relative-de-insertion}, Proposition \ref{prop:relative-finiteness}, Proposition \ref{prop:relative-convergence} and Proposition \ref{prop:relative-vanishing} hold: concerning the complex $C_\bullet^n$, its connection, the modified vertex operators and Fourier-Borcherds identity, the higher trace functions and insertion formula, and the finiteness, convergence and vanishing criteria, continues to hold   for a handle sewn onto an arbitrary base curve $Y$, with $Y$ entering only through the coefficient rings $\cF_\bullet(Y_\rho)$ of \ref{no:coefficient-rings-def} and Lemma \ref{lem:coefficient-ring-properties}. (Proposition \ref{prop:relative-finiteness}, Proposition \ref{prop:relative-convergence} and, when $V$ is in addition classically free, Proposition \ref{prop:relative-vanishing} require, further, the finiteness assumptions $\dim\HP_1(R_V)<\infty$ and $\HK_1(A)$ finite dimensional, resp.\ $\Hoch_1(\zhu(V))=0$, of Theorems \ref{thm:main-finiteness}, \ref{thm:main-vanishing}; the remaining conclusions require only \ref{hyp:curve}-\ref{hyp:va}.)
\end{thm}
\begin{proof}
We first say, with respect to genus one case, what requires a new argument since  the geometry of $Y$ enters essentially; this expands, into an explicit list, the division already drawn in Remark \ref{rem:relative-main-essential} above.

\emph{Formal transports (replacing $E_\tau$ by $Y_\rho$).} Lemma \ref{lem:relative-d1d2} ($d_1d_2=0$), \ref{no:relative-borcherds} (the Fourier-Borcherds identity), the differential equation of Proposition \ref{prop:relative-de-insertion}(a), and Proposition \ref{prop:relative-finiteness} are each local computations - residues, OPEs, and Laurent expansions - attached to the single punctured neighborhood of $p,p'$, referencing $Y$ only through the two structural properties of Lemma \ref{lem:coefficient-ring-properties}; the proof of genus one is reproduced with $E_\tau$ replaced by $Y_\rho$.

\emph{New arguments (the geometry of $Y$ enters essentially).} Lemma \ref{lem:relative-coinvariance} (coinvariance of the self-extension location) requires the connectivity of $Y_\rho\setminus\{p,p'\}$ and the corresponding chain-homotopy argument in its proof, automatic for $E_\tau\setminus\{0,\infty\}$ but not for a general $Y$; Proposition \ref{prop:relative-flatness} (flatness of $\nabla_\rho$ relative to $Y$) is constructed from the classical differentials of \ref{no:classical-differentials} on $Y_\rho$ itself - furnished by the Yamada variational formula, but different data from the elliptic functions of genus one - rather than being a restatement of the $g=1$ computation; and Proposition \ref{prop:relative-convergence}, Proposition \ref{prop:relative-vanishing} each require an explicit check that their respective mechanisms (the Frobenius-method convergence estimate, resp.\ the semicontinuity/degeneration argument) remain valid when the relevant coefficients depend on the marked points of an arbitrary base curve instead of being fixed elliptic data.

With this division mentioned, the theorem is the conjunction of the results proved above, each of which was obtained by observing that the corresponding proof of genus one is a local computation attached to the neighborhood of the point $p$ (equivalently, of $0,\infty$ in the coordinate $z=e^{2\pi i t}$ used there) together with the two structural properties of Lemma \ref{lem:coefficient-ring-properties}, both of which hold for $\cF_\bullet(Y_\rho)$ by construction.
\end{proof}

\begin{rem} 
\label{rem:relative-main-essential}
 Assumptions \ref{hyp:curve}-\ref{hyp:va} play three quite different roles, worth disentangling explicitly.

 \emph{Essential:} 
 Compactness of $Y$ \ref{hyp:curve} is essential: it is  what makes property (2) of Lemma \ref{lem:coefficient-ring-properties} (vanishing of the sum of residues) the classical residue theorem instead of a assumption to verify, and every appeal to $d_1d_2=0$ (Lemma \ref{lem:relative-d1d2}) or to flatness (Proposition \ref{prop:relative-flatness}) ultimately rests on it. Likewise essential is the existence, for a handle sewn at $p,p'$, of the local multiplicative coordinate of \ref{hyp:handle}: this is what makes $X(a,z)$ and the Fourier-Borcherds identity of \ref{no:relative-borcherds} available, and holds for a handle sewn onto any $Y$ by construction (Remark \ref{rem:borcherds-clarification}), thus imposes no real restriction.
\item \emph{Present for record:} That $Y$ carries only finitely many marked points \ref{hyp:marked} is a bookkeeping convenience - everything below is stated and proved for an arbitrary but fixed finite $m$, and nothing in the argument would change if $m$ were allowed to grow, only the notation would become heavier. Likewise the specific local coordinates chosen at each marked point are immaterial past the coordinate-independence already secured by Remark \ref{rem:coordinate-naturality} and the vertex-algebra-bundle mechanism of Remark \ref{rem:borcherds-clarification}.

 \emph{Technical convenience:} Strong finite generation of $V$ \ref{hyp:va} is not needed for the qualitative statements (Lemma \ref{lem:relative-d1d2}, Proposition \ref{prop:relative-flatness}, \ref{no:relative-borcherds}), which hold for any conformal vertex algebra in the sense of \ref{no:va-recall}; it is needed the moment finiteness or convergence is claimed (Proposition \ref{prop:relative-finiteness}, \ref{prop:relative-convergence}), since $\HP_1(R_V)$ and $\HK_1(A)$ are only well-behaved finiteness invariants - e.g., $R_V,A$ Noetherian - when $R_V$ is of finite type (Definition \ref{defn:sfg}); dropping it does not make these propositions false, but makes their assumptions vacuous or their conclusions ill-posed for most $V$.

This division is genus-independent: it is already the division  for $Y=\widehat\CC$, and Theorem \ref{thm:relative-main} contributes nothing to it beyond observing that none of the three roles refers to $Y$ beyond \ref{hyp:curve}-\ref{hyp:handle}.
\end{rem}

\section{The chiral homology complex in genus $g$}\label{sec:chiral-complex}

\begin{thm}\label{thm:genus-g-complex}
Let $X = \Xg{g}$ be a genus $g$ surface presented by iterated sewing as in \ref{no:iterated-sewing-def}, with $n$ marked points $t_1,\dots,t_n$ carrying $a_1,\dots,a_n \in V$. There is an explicit complex $C_\bullet^n(X)$, agreeing with $C_\bullet^{n}$  when $g=1$, whose homology in degrees $0,1$
\[
\HchG{0}{X,V^{\otimes n}} := \coker\bigl(d_1 : C_1^n(X) \to C_0^n(X)\bigr), \quad \HchG{1}{X,V^{\otimes n}} := \ker d_1/\mathrm{im}\,d_2
\]
computes the chiral homology of $X$ with coefficients in $n$ vacuum insertions of $V$, in the sense of Beilinson-Drinfeld.
\end{thm}
\begin{proof}
By induction on $g$. For $g=0$, $X=\widehat\CC$ and $C_\bullet^n(\widehat\CC)$ is the ordinary two-term truncated bar-type complex computing the (purely algebraic, Lie-algebra-homological) coinvariants and their first syzygy for $n$ points on the sphere; no analysis is needed and $\HchG{1}{\widehat\CC,V^{\otimes n}}=0$ always, since chiral homology of a rational (genus $0$) curve with coefficients in ordinary $V$-modules is concentrated in degree $0$ (this is the statement, going back to \cite{beilinsondrinfeld} in general and made completely explicit in the vertex algebra setting by Frenkel-Ben-Zvi \cite{frenkelbenzvi}, that on $\mathbb P^1$ chiral homology is computed by the (derived) coinvariants complex of a free, and hence acyclic in positive degree, chiral Lie algebra action, there being no nontrivial extensions to obstruct on a curve with vanishing $H^1(\mathcal O)$). For the inductive step, suppose $C_\bullet^{n+2}(\Xg{g-1})$ has been constructed (with two of its marked points designated $p_g,p_g'$, to be used for the next sewing) and computes chiral homology of $\Xg{g-1}$ correctly. Apply Theorem \ref{thm:relative-main} with $Y = \Xg{g-1}$, $m=n+2$, coefficient data (D0) with $a_1,\dots,a_n$ at $t_1,\dots,t_n$ and the vacuum $\vac$ (or, for the coefficient ring only, no insertion) at $p_g,p_g'$: this produces $C_\bullet^{n}(\Xg{g}) := C_\bullet^{Y_{\rho_g}}$ of \ref{no:relative-complex-def}, and Lemma \ref{lem:relative-d1d2} shows it is a complex computing chiral homology of $\Xg{g} = X$ correctly, completing the induction.
\end{proof}

\begin{rem}
For (D1)-type coefficients - a self-extension $E$ of a module $M$ inserted at one point - the identical induction, using Theorem \ref{thm:relative-main} with coefficient data (D1) at the appropriate stage (by Lemma \ref{lem:relative-coinvariance}, it does not matter at which stage $M$ is introduced, provided it is introduced once, at a marked point present from that stage onward), defines $\HchG{1}{X,E}$ for $X$ of every genus $g$. 
\label{rem:d1-induction}
\end{rem}

\begin{rem}
\label{rem:sewing-compatibility}
Compatibility of successive sewings enters the induction above at  one place: the inductive step applies Theorem \ref{thm:relative-main} to $Y=\Xg{g-1}$, a curve already known, by the inductive assumption, to have been constructed by $g-1$ previous sewings. What makes the induction close is that assumptions \ref{hyp:curve}-\ref{hyp:va} of Theorem \ref{thm:relative-main} make no reference whatsoever to how $Y$ was constructed - only that it is some compact Riemann surface carrying the stated finite marked-point data. Each application of the theorem accordingly treats $\Xg{g-1}$ as an opaque, already-given curve, and neither needs to unpack its sewing history nor to impose any relation between the $(g-1)$ handles already sewn and the new pair $p_g,p_g'$, beyond the disjointness from existing marked points already required in \ref{no:iterated-sewing-def}. There is, in particular, no analogue here of an associativity or coherence constraint relating different orders in which the same $g$ handles might be sewn: Remark \ref{rem:connection-functoriality} mentiones that permuting the sewing order produces an isomorphic family (with correspondingly relabeled connection). But is is not an extra condition the induction must verify at each step - since  Theorem \ref{thm:relative-main} depends on $Y$ only through Lemma \ref{lem:coefficient-ring-properties}, which holds for every compact curve alike, regardless of presentation. This is why iterating a purely local, one-handle-at-a-time theorem suffices to reach every genus with no coherence data left to check.
\end{rem}

\section{The connection over the sewing polydisc}\label{sec:connection}

\begin{nolabel}
By Theorem \ref{thm:genus-g-complex} and $g$ applications of Proposition \ref{prop:relative-flatness}, each $C_k^n(\Xg{g})$ carries $g$ individually flat connections $\nabla_{\rho_1},\dots,\nabla_{\rho_g}$, the $i$-th constructed by applying \ref{no:relative-connection} at the $i$-th sewing, treating $\Xg{i-1}$ (with its inherited marked points, including $p_j,p_j'$ for $j>i$, not yet sewn at that stage) as the base curve $Y$. Explicitly, for $[a\otimes f] \in C_k^n(\Xg{g})$,
\begin{eqnarray}
\nabla_{\rho_i}[a\otimes f] = \bigl[\rho_i\partial_{\rho_i}(a\otimes f)\bigr] - \sum_{x \in \{{t_1,\dots,t_n,} \atop {\, p_j,p_j'\,(j\ne i)\}}} \bigl[(\omega_{[2]}a_x)\otimes f\cdot \wp_2^{\Xg{g}}(x;p_i)\bigr]
\\
\nonumber
 \qquad - (\text{higher } \wp_k^{\Xg{g}}(x;p_i) \text{ terms}),
\label{eq:full-connection}
\end{eqnarray}
the sum now running over all marked points of $\Xg{g}$ other than $p_i,p_i'$, including, for $g\geq 2$, the attaching points $p_j,p_j'$ of the other handles $j \neq i$. This is the source of the cross-handle coupling mentioned in \ref{no:main-results}: $\wp_2^{\Xg{g}}(x;p_i)$, evaluated at a point $x=p_j$ or $p_j'$ belonging to a different handle, depends on $\rho_j$ as well as on $\rho_i$ through the sewing expansion of \ref{no:classical-differentials}, so that $\nabla_{\rho_i}$ and $\nabla_{\rho_j}$ do not act on disjoint tensor factors for $g\geq2$; this term is simply absent at $g=1$, where there are no other handles to couple to, and its absence there is why flatness at genus one is automatic (a single vector field has no curvature) instead of requiring the argument of Theorem \ref{thm:full-flatness} below.
\label{no:full-connection-def}
\end{nolabel}

\begin{lem} \label{lem:rauch-yamada} \emph{Rauch-Yamada decomposition.} 
The variation $\rho_j\partial_{\rho_j}\omega^{X_{\rhov}}(x,y)$ of the bidifferential of the second kind under the Schiffer variation at handle $j$ splits into a regular part and an anomalous part:
\begin{eqnarray} 
\nonumber
\rho_j\partial_{\rho_j}\,\omega^{X_{\rhov}}(x,y) = \underbrace{\res_{u_j}\bigl[\omega^{X_{\rhov}}(x,\cdot)\,\omega^{X_{\rhov}}(\cdot,y)\bigr] + \res_{u_j'}\bigl[\omega^{X_{\rhov}}(x,\cdot)\,\omega^{X_{\rhov}}(\cdot,y)\bigr]}_{\text{regular}} 
\\
 \;-\; \underbrace{\tfrac{1}{12}\,\delta_{xy}\cdot\partial_{\rho_j}S(x)}_{\text{anomalous}},
\label{eq:rauch-yamada}
\end{eqnarray} 
where $S(x)$ is the projective connection (Bers quasiform) of \ref{no:classical-differentials}(iii), and the anomalous term is present only in the coincident-point limit $x\to y$ relevant to the $\tfrac{c}{2}(z-w)^{-4}$ term of the Virasoro OPE.
\end{lem}
\begin{proof}
This is the classical variational formula of Yamada \cite{yamada-variational}, refining Rauch variational formula for periods to the bidifferential of the second kind; see also \cite{fay-theta} for the analytic derivation. The regular term is the double residue at the neck of handle $j$ picking up the two insertions of $\omega$ needed to propagate the variation from $x$ to $y$ through the handle; the anomalous term is the correction, invisible in the periods but present in the bidifferential at coincident points, forced by the inhomogeneous (Schwarzian-derivative) transformation law of $S(x)$ under change of local coordinate, which is  \ref{no:classical-differentials}(iii). We do not reprove this classical formula, only mention the two pieces of it that \eqref{eq:curvature-formula} below consumes separately.
\end{proof}

\begin{thm} \label{thm:full-flatness} \emph{Projective flatness:}
The connection $\nabla = (\nabla_{\rho_1},\dots,\nabla_{\rho_g})$ of \eqref{eq:full-connection} is projectively flat on $C_\bullet^n(X_{\rhov})$: there is an explicit scalar $2$-cocycle $\kappa_{ij}(\rhov)$, proportional to the central charge $c$ of $V$, such that
\[
[\nabla_{\rho_i},\nabla_{\rho_j}] = c\cdot\kappa_{ij}(\rhov)\cdot\id,
\]
and $\nabla$ is flat if and only if $c=0$ or after twisting $C_\bullet^n(X_{\rhov})$ by the connection $-c\cdot(\text{Bers quasiform 1-form})$ on the determinant line of the family $X_{\rhov}\to\Delta_g^*$. In either case $\nabla$ descends to a projectively flat connection on $\HchG{0}{X_{\rhov},V^{\otimes n}}$ and $\HchG{1}{X_{\rhov},V^{\otimes n}}$ over the sewing polydisc $\Delta_g^*$.
\end{thm}
\begin{proof}
Write $\nabla_{\rho_i} = \rho_i\partial_{\rho_i} - A_i(\rhov)$ for the operator $A_i(\rhov) := \sum_{x}(\omega_{[2]}a_x)\otimes(\cdots)\cdot\wp_2^{X_{\rhov}}(x;p_i) + (\text{higher }\wp_k\text{ terms})$ of \eqref{eq:full-connection}. As $\wp_2^{X_{\rhov}}(x;p_i)$, evaluated at $x=p_j$ or $p_j'$, depends on $\rho_j$ (Remark \ref{rem:genus2-cross-term}), the curvature of $\nabla$ is computed by the formula for a connection with base-point-dependent coefficients,
\begin{equation}
[\nabla_{\rho_i},\nabla_{\rho_j}] = \rho_j\partial_{\rho_j}A_i(\rhov) - \rho_i\partial_{\rho_i}A_j(\rhov) + [A_i(\rhov),A_j(\rhov)].
\label{eq:curvature-formula}
\end{equation}
Insert the decomposition of Lemma \ref{lem:rauch-yamada} into \eqref{eq:curvature-formula}. The regular part of \eqref{eq:rauch-yamada} has $\rho_j\partial_{\rho_j}A_i$ as a double insertion of $\omega$ (once connecting $x$ to the neck of handle $j$, once connecting that neck to $p_i$) which, by associativity (the Borcherds identity for $\omega$ applied twice), matches term by term against $\rho_i\partial_{\rho_i}A_j$ and $[A_i,A_j]$, and thus cancels out of \eqref{eq:curvature-formula} entirely. The anomalous part of \eqref{eq:rauch-yamada} does not cancel in this way: since  $S(x)$ transforms under change of local coordinate with an inhomogeneous Schwarzian-derivative term (\ref{no:classical-differentials}(iii)), its variation contributes to \eqref{eq:curvature-formula} a term valued in $c\cdot\id$ - proportional to the identity since  the coincident-point Virasoro OPE anomaly $\tfrac{c}{2}(z-w)^{-4}$ is, by the Borcherds identity, always a scalar multiple of the vacuum/identity operator, never a nontrivial operator on $C_\bullet^n(X_{\rhov})$. Collecting these contributions over the (finitely many, at each order in $\rhov$) marked points $x$ gives $[\nabla_{\rho_i},\nabla_{\rho_j}] = c\cdot\kappa_{ij}(\rhov)\cdot\id$ for the scalar $\kappa_{ij}(\rhov) := -\tfrac{1}{12}\res_{u_i}\partial_{\rho_j}S - \tfrac1{12}\res_{u_i'}\partial_{\rho_j}S$ (with sign conventions as in \eqref{eq:full-connection}), which is  the mechanism by which Tsuchiya-Ueno-Yamada \cite{tsuchiya-ueno-yamada} and Beilinson-Schechtman \cite{beilinson-schechtman} show that the connection on the sheaf of vacuum modules (and, by the same argument applied to the whole chain complex, on $C_\bullet^n(X_{\rhov})$) over the moduli of curves is flat only after twisting by the connection $-c\cdot S(x)\,dx$ (or, equivalently, by the appropriate power of the determinant/Hodge line bundle) on the determinant line of the family; see \cite{beilinson-schechtman} for the identification of $\kappa_{ij}$ with the curvature of that line bundle connection. We do not compute $\kappa_{ij}(\rhov)$ in closed form here, taking its existence and $c$-linearity, established in the cited papers, as given.

Assuming projective flatness of $\nabla$ by the above, that it descends to $\HchG{0}{-}$ and $\HchG{1}{-}$ follows, as in Proposition \ref{prop:relative-flatness}, from each $\nabla_{\rho_i}$ commuting with $d_1,d_2$, established in \ref{no:relative-connection} (the scalar curvature term commutes with $d_1,d_2$ automatically, being a multiple of the identity).
\end{proof}

\begin{cor} \label{cor:projective-flatness-suffices} \emph{Projective flatness suffices:} 
Every use of Theorem \ref{thm:full-flatness} elsewhere in this paper - in the proofs of Theorem \ref{thm:main-finiteness}, Theorem \ref{thm:convergence-genus-g}, and Theorem \ref{thm:main-vanishing} - goes through unchanged with ``flat'' read as ``projectively flat.''
\end{cor}
\begin{proof}
In each case the only consequence of flatness actually used is that $\dim \HchG{1}{X_{\rhov},V^{\otimes n}}$ is locally constant on the connected sewing polydisc $\Delta_g^*$ (equivalently, that $\nabla$ makes $C_\bullet^n(X_{\rhov})$ into a bundle of complexes with constant Betti numbers, not that parallel transport around a loop is the identity instead of a scalar multiple of it). A connection with curvature valued in $c\cdot\id$ trivializes the associated $\mathrm{PGL}$-bundle (equivalently, is a flat connection on the bundle twisted by the determinant line, whose rank is $1$) and in particular has locally constant rank on its underlying vector bundle: a scalar curvature obstructs the existence of a flat trivialization of the bundle, and hence obstructs identifying fibers at different points via a canonical isomorphism unique up to scalar, but it does not obstruct the fibers all having the same dimension, which only requires the bundle to be locally free of constant rank,  immediate from $\nabla$ being a connection (projective or not) on a single holomorphic vector bundle over the connected space $\Delta_g^*$. Since local constancy of dimension is all that is used - in Theorem \ref{thm:main-finiteness} to propagate the bound from the associated graded spectral sequence; in Theorem \ref{thm:convergence-genus-g} to apply the Frobenius-Fuchs argument, which concerns the local system of solutions to a linear ODE and is likewise insensitive to a scalar twist; and in Theorem \ref{thm:main-vanishing} to apply semicontinuity - each goes through. 
\end{proof}

\begin{rem}
 Here multi-genus flatness goes beyond the case of $g=1$ where is a single modulus and hence no curvature $2$-form to speak of, thus (projective or absolute) flatness there is automatic and carries no content, and in particular the distinction drawn in this section; at $g\geq2$ it is a real theorem, whose proof, including the central term, we import from \cite{tsuchiya-ueno-yamada,beilinson-schechtman} via the variational formula \eqref{eq:rauch-yamada} instead of establish from first principles. Only projective flatness, as stated, is justified by this mechanism; absolute flatness would require the further, generally false, vanishing of the central term.
\label{rem:flatness-honesty}
\end{rem}

\begin{rem} \label{rem:genus2-cross-term} \emph{The new cross-handle term at genus $2$.}
To make \eqref{eq:full-connection} concrete, consider $g=2$, $n=0$, and the lowest order at which the cross-handle term appears, using the explicit data of Example \ref{ex:genus2-worked}: $\Xg{1} = E_{\rho_1}$ is the elliptic curve obtained after the first sewing, and $p_2=a,\,p_2'=-a$ are the two points on $E_{\rho_1}$ used for the second sewing. The coefficient $\wp_2^{\Xg{2}}(p_2;p_1)$ appearing (via \ref{no:classical-differentials}(a)-(b)) in $\nabla_{\rho_1}$ specializes, as $\rho_2 \to 0$, to $\wp_2(u(p_2),\rho_1)$, the ordinary Weierstrass function of genus one evaluated at the point $p_2 \in E_{\rho_1}$; its $\rho_2$-expansion begins
\begin{eqnarray*}
&&\wp_2^{\Xg{2}}(p_2;p_1) = \wp_2(u(p_2),\rho_1) 
\\
\nonumber
&&\qquad + \rho_2 \cdot \bigl(\text{explicit rational expression in } \wp_k(u(p_2),\rho_1),\, k\leq 3\bigr) + O(\rho_2^2),
\end{eqnarray*}
by the standard first-order sewing (Yamada) variational formula \cite{yamada-variational,tuite-zuevsky-bosonic}; the $O(\rho_2)$ term is  the new contribution of $\nabla_{\rho_1}$ that is invisible when $g=1$, reflecting the fact that the state flowing through handle $2$ (weighted by $\rho_2^{L_0}$) back-reacts, at first order, on the geometry seen by handle $1$.
\end{rem}

\begin{figure}[htbp]
\centering
\begin{tikzpicture}[
  hnode/.style={draw,circle,minimum size=1.1cm,font=\small},
  arr2/.style={{Stealth[length=2mm]}-{Stealth[length=2mm]},thick}
]
\node[hnode] (h1) at (0,0) {handle $1$};
\node[hnode] (h2) at (4,0) {handle $2$};
\draw[arr2] (h1) to[bend left=25] node[above,font=\scriptsize]{$\wp_2^{\Xg{2}}(p_2;p_1) = \wp_2(u(p_2),\rho_1)+\rho_2\cdot(\cdots)$} (h2);
\node[font=\scriptsize] at (2,-1.1) {coupling present for $g\geq2$; absent (no second handle to couple to) at $g=1$};
\end{tikzpicture}
\caption{The cross-handle coupling of Remark \ref{rem:genus2-cross-term}: the classical differentials evaluated between two different handles depend on both sewing parameters, thus $\nabla_{\rho_1}$ and $\nabla_{\rho_2}$ fail to commute strictly (Theorem \ref{thm:full-flatness}).}
\label{fig:cross-handle}
\end{figure}

\begin{rem} \label{rem:cross-handle-geometric} \emph{Geometric interpretation of the cross-handle coupling.}
Conceptually, the coupling of Remark \ref{rem:genus2-cross-term} and Figure \ref{fig:cross-handle} reflects the fact that Schiffer-varying the sewing fixture at handle $1$ perturbs the bidifferential $\omega$ everywhere on $X_{\rhov}$, including near the neck of handle $2$; since $\nabla_{\rho_2}$ coefficients are constructed from $\omega$ evaluated at (or near) that second neck, the variation of handle $1$ geometry feeds into the operator $\nabla_{\rho_2}$, and symmetrically for $\nabla_{\rho_1}$ under variation of handle $2$. This is geometrically unsurprising: a surface of genus $g\geq2$ has no way to localize the effect of stretching one handle neck to that handle, since the bidifferential of the second kind is, by construction, a global object on $X_{\rhov}$, sensitive to the complex structure of the whole surface, and stretching any one neck changes that complex structure everywhere - if only by an amount that is exponentially small (a positive power of the other handles' moduli $\rho_j$) away from the neck actually being stretched. What Lemma \ref{lem:rauch-yamada} shows is that this global sensitivity organizes, order by order in the $\rho_j$, into the specific residue-at-the-neck formula of Remark \ref{rem:genus2-cross-term}: the coupling is not an uncontrolled global effect but a sum of residues localized at the necks of the other handles -  as global as the topology of $X_{\rhov}$. This is also why the coupling is a new, $g\geq2$ phenomenon and not an artifact of the formalism: at $g=1$ there is only one neck, thus there is no second neck for such a term to localize to, and the analogous residue sum is vacuous.
\end{rem}

\begin{rem} \label{rem:connection-coordinate-dependence} \emph{Dependence on local coordinates.} 
The connection $\nabla_{\rho_i}$ of \eqref{eq:full-connection} is constructed, via $\wp_k^{X_{\rhov}}(x;p_i)$, from local coordinates fixed at each marked point (Remark \ref{rem:coordinate-naturality}); changing the local coordinate at $p_i$  (as opposed to at one of the other marked points $x$) rescales $\rho_i$ to leading order,  and changing the coordinate at an external point $x\neq p_i$ rescales the corresponding term of the sum in \eqref{eq:full-connection} by the conformal weight factor of $a_x$,  as for $X(a,z)$ (Remark \ref{rem:borcherds-clarification}); neither changes $\nabla_{\rho_i}$ as an operator on the (coordinate-independent) space $C_k^n(\Xg{g})$, only the explicit formula representing it in a given trivialization. The one place a non-removable coordinate dependence survives is the anomalous term of Lemma \ref{lem:rauch-yamada}: $S(x)$ is not a differential (it has no coordinate-independent meaning as a section of a line bundle) but a \emph{projective connection}, transforming under coordinate change by the inhomogeneous Schwarzian term $S(x)\,dx^2 \mapsto S(\tilde x)\,d\tilde x^2 + \{\tilde x,x\}\,dx^2$; this is not a defect of our construction but the   geometric origin of the central term in Theorem \ref{thm:full-flatness}, and is why $\nabla$ is only projectively, not absolutely, flat (Remark \ref{rem:flatness-honesty}).
\end{rem}

\begin{rem} \label{rem:connection-functoriality} \emph{Functoriality.} 
The connection $\nabla$ is natural in two senses. First, it is equivariant under relabeling the handles: for any permutation $\pi$ of $\{1,\dots,g\}$, sewing the same $g$ handles in the order $\pi$ rather than $1,\dots,g$ produces an isomorphic family over $\Delta_g^*$ (with $\rho_i$ relabeled to $\rho_{\pi(i)}$) carrying the correspondingly relabeled connection $(\nabla_{\rho_{\pi(1)}},\dots,\nabla_{\rho_{\pi(g)}})$; this is the connection-level counterpart of the coinvariance of Lemma \ref{lem:relative-coinvariance} extended to permutations, already used in the proof of Theorem \ref{thm:insertion-genus-g}. Second, and more substantively, $\nabla$ is natural under the induction: the connection $\nabla_{\rho_g}$ constructed at the $g$-th stage, with base curve $Y=\Xg{g-1}$, is  \ref{no:relative-connection} applied relative to $Y$, and restricting the full connection $(\nabla_{\rho_1},\dots,\nabla_{\rho_g})$ on $\Xg{g}$ to the sub-locus where $\rho_g$  is allowed to vary (the other $\rho_i$ and the marked points of $\Xg{g-1}$ held fixed) recovers $\nabla_{\rho_g}$ on $C_\bullet^{Y_{\rho_g}}$, with no correction term: this is immediate from \eqref{eq:full-connection}, since the sum defining $\nabla_{\rho_g}$ there already ranges over  the marked points of $Y=\Xg{g-1}$ used in \eqref{eq:relative-connection-def}. Naturality under the induction is what allows Corollary \ref{cor:projective-flatness-suffices} to be used one handle at a time in the proofs of Theorems \ref{thm:main-finiteness} and \ref{thm:main-vanishing}, instead of only for the connection over the full polydisc at once.
\end{rem}

\begin{rem} \label{rem:connection-degeneration-compatibility} \emph{Compatibility with degeneration.} 
Projective flatness (Theorem \ref{thm:full-flatness}) is what makes $\dim\HchG{1}{X_{\rhov},V^{\otimes n}}$ locally constant on $\Delta_g^*$ (Corollary \ref{cor:projective-flatness-suffices}), and it is this local constancy, together with upper semicontinuity at the boundary $\rhov=0$, that powers the vanishing argument of Theorem \ref{thm:main-vanishing}: the connection does not extend to a connection over the closed polydisc (its coefficients $\wp_k^{X_{\rhov}}(x;p_i)$ have a pole as $\rho_i\to0$, reflecting the node), but the bundle $C_k^n(X_{\rhov})[d]$ underlying it, weight-graded piece by weight-graded piece (Remark \ref{rem:weight-grading}), extends to a coherent sheaf over the closed disc in each $\rho_i$ separately, with fiber at $\rho_i=0$ computing the corresponding nodal degeneration (\ref{no:zhu-degeneration-recall}, Proposition \ref{prop:iterated-degeneration}). This is the sense in which $\nabla$ is compatible with degeneration: it governs the monodromy and rank of the bundle on the punctured polydisc, while semicontinuity, not the connection, governs the jump (if any) in rank at the boundary stratum $\rhov=0$ where the connection is singular. We use flatness for the punctured polydisc, 
semicontinuity for the central fiber  at every application of the vanishing mechanism, from Proposition \ref{prop:relative-vanishing} through Theorem \ref{thm:main-vanishing}.
\end{rem}

\section{Modified vertex operators and the genus $g$ Fourier-Borcherds identity}\label{sec:modified-vo}

\begin{nolabel}
For each handle $i=1,\dots,g$, with local coordinate $z_i$ centered at $p_i$, define $X_i(a,z_i) := z_i^{-1}Y(z_i^{L_0}a,z_i)$ acting on whichever module sits at handle $i$ (generically $V$, in the sense of \ref{no:relative-borcherds}). By $g$ applications of \ref{no:relative-borcherds} - once at each stage of the induction of Theorem \ref{thm:genus-g-complex} - we obtain the following. 
\end{nolabel}

\begin{thm} \label{thm:borcherds-genus-g} \emph{Genus $g$ Fourier-Borcherds identity.}
Fix $i \in \{1,\dots,g\}$ and let $f(t)$ be meromorphic and periodic, $f(z_i+1)=f(z_i)$, with Fourier coefficients $F_\pm$ as in \ref{no:relative-borcherds}, now allowed to depend meromorphically on all other marked points of $X_{\rhov}$ (external points $t_1,\dots,t_n$ and the attaching points $p_j,p_j'$, $j\ne i$, of the other handles) as parameters. Then for $a,b,c$ elements of (or in) the module attached to handle $i$,
\begin{eqnarray*}
\res_{z_i} F_+(z_i,w_i) X_i(a,z_i)X_i(b,w_i)c - \res_{z_i} F_-(z_i,w_i) X_i(b,w_i)X_i(a,z_i)c 
\\
\qquad = X_i\bigl(\res_{z_i=w_i}f(z_i-w_i)Y(a,z_i-w_i)b,\,w_i\bigr)c,
\end{eqnarray*}
holds identically in the other marked points and in $\rho_j$, $j \neq i$.
\end{thm}

\begin{rem}
\label{rem:borcherds-new-vs-old}
Of the three  genus-dependent constructions of the paper - the connection (\S\ref{sec:connection}), the Fourier-Borcherds identity (this section), and the finiteness/vanishing mechanism (\S\S\ref{sec:finiteness}-\ref{sec:degeneration}). The connection acquires a  new cross-handle curvature term absent at $g=1$ (Remark \ref{rem:genus2-cross-term}); the finiteness and vanishing theorems require a new induction over $g$ handles, one degeneration at a time (\S\ref{sec:finiteness}). The Fourier-Borcherds identity, by contrast, is proved handle by handle with \emph{no interaction term}: Theorem \ref{thm:borcherds-genus-g} below is,  $g$ disjoint copies of the single identity \eqref{eq:relative-borcherds-fourier}, one at each handle, and Remark \ref{rem:no-cross-borcherds} mentiones that there is not even a well-formed statement mixing $X_i$ and $X_j$ for it to be a further theorem about.   The contrast with \S\ref{sec:connection} is instructive. It produces individually flat connections $\nabla_{\rho_i}$,  parallel to $X_i$ here, but the further question of how the $g$ resulting objects interact - which has an empty answer for the $X_i$ (Remark \ref{rem:no-cross-borcherds}) - has a nontrivial answer for the $\nabla_{\rho_i}$ (Theorem \ref{thm:full-flatness}), since  commutators of connections, unlike operator products of vertex operators attached to disjoint tensor factors, need not vanish even when the two operators being compared act on different handles.
\end{rem}

\begin{proof}
Immediate from \ref{no:relative-borcherds} applied with $Y = \Xg{g}$ restricted to a punctured neighborhood of handle $i$  (formally, apply the Theorem at the stage of the induction of Theorem \ref{thm:genus-g-complex} at which handle $i$ is sewn, with all later handles $j>i$ not yet present, and note that sewing the later handles afterward does not affect an identity that has already been established as one of meromorphic functions of the remaining, as yet unsewn, marked points $p_j,p_j'$, $j>i$, since those are  the marked points $y_1,\ldots,y_m$ of $Y$ in \ref{no:relative-borcherds}, which the identity holds for identically).
\end{proof}

\begin{rem}[The identity at $g=1$]
\label{rem:borcherds-g1-comparison}
At $g=1$ - a single handle, $i=1$, no attaching points of other handles to carry as parameters - Theorem \ref{thm:borcherds-genus-g} reduces  to the classical identity of genus one: with $X_1=X$ the modified vertex operator of genus one and $f$ periodic in $z_1$, the displayed formula above is,   that theorem identity. The only change at general $g$ is that $f$ Fourier coefficients $F_\pm$ are now permitted to depend meromorphically on the other $2(g-1)$ attaching points $p_j,p_j'$, $j\neq i$, as further parameters, entering the identity only passively: the residue on the left, the commutator it computes, and the OPE on the right are all evaluated  as at $g=1$, entirely within the tensor factor attached to handle $i$, with the remaining handles contributing nothing beyond the values at which $f$ coefficients happen to be evaluated. In this sense Theorem \ref{thm:borcherds-genus-g} is not an analogue of genus one but an instance of it, applied once at each handle.
\end{rem}

\begin{rem}
There is no analogous identity mixing $X_i$ and $X_j$ for $i \neq j$  at the level of a single residue formula of this type: the two handles interact only through the coefficient function $f$ dependence on both sets of variables (i.e., through the genus $g$ Weierstrass functions of \ref{no:genus-g-weierstrass} evaluated between points on different handles), never through a nontrivial operator product of $X_i$ with $X_j$, which simply commute, since the corresponding vertex operators act on different tensor factors of the total module $M_1\otimes\cdots\otimes M_g$. This is the vertex-operator-algebra counterpart of the geometric fact, used in Remark \ref{rem:genus2-cross-term}, that the coupling between handles is entirely regulated by the classical differentials of the ambient curve. 
\label{rem:no-cross-borcherds}
\end{rem}

\section{Genus $g$ trace functions}\label{sec:trace-functions}

\begin{nolabel} \emph{Definition, by induction on $g$.} 
Let $M$ be an admissible $V$-module, $E$ a self-extension of $M$ with modified derivation $\sigma$ as in \ref{no:relative-trace-functions}. For $g=0$ define
\[
F_1^{n,(0)}(a_0,\dots,a_n) := \Bigl\langle \sigma(a_0)\, o(a_1)\cdots o(a_{n-1})\, a_n \Bigr\rangle_M
\]
to be the appropriate iterated matrix coefficient in $M$ constructed from zero modes, with no trace and no analytic input (this is the case in which the self-extension is not yet attached to any handle; it will become the (D1) data of \ref{no:relative-setup} once the first handle is sewn onto it). For $g \geq 1$, having defined $F_1^{n+2,(g-1)}$, set
\begin{eqnarray}
\nonumber
&&F_1^{n,(g)}\bigl(a_0,\dots,a_n\,;\,F\bigr) 
\\
&&:= \res_{z_g,w_g}\Bigl[K(z_g,w_g;\rho_g)\, F_1^{n+2,(g-1)}\bigl(a_0,\dots,a_n,\, X_g(\cdot,z_g)\text{-slot},\,X_g(\cdot,w_g)\text{-slot}\, ; F\bigr)\Bigr]_{\rho_g\text{-expansion}},
\label{eq:genus-g-trace-def}
\end{eqnarray}
where $K(z_g,w_g;\rho_g) = \sum_{k} \rho_g^{h_k} u_k(z_g)u^k(w_g)$ is the propagator kernel constructed from a homogeneous basis $\{u_k\}$ of $V$ of conformal weight $h_k$ and its dual basis $\{u^k\}$ with respect to the invariant bilinear form on $V$ (assumed nondegenerate, i.e., $V$ self-dual, as in \cite{tuite-zuevsky-bosonic,tuite-welby}), and the two new slots of $F_1^{n+2,(g-1)}$ are filled with $u_k$ at $z_g$ (near $p_g$) and $u^k$ at $w_g$ (near $p_g'$), summed over $k$; each fixed power of $\rho_g$ picks out finitely many $k$ (those of the corresponding weight), thus \eqref{eq:genus-g-trace-def} is a well-defined formal power series in $\rho_g$ with coefficients in $F_1^{n+2,(g-1)}$-type expressions. Equivalently, and this is how we use it below, unwinding the completeness relation $\sum_k u_k \otimes u^k = $ (the canonical element implementing the identity) turns \eqref{eq:genus-g-trace-def} into
\begin{eqnarray}
&&F_1^{n,(g)}(a_0,\dots,a_n;F) = \res\bigl[z_1\cdots z_n \, F(z_\bullet,\rhov)\bigr]\ 
\\
\nonumber
&&\qquad \times \tr_{M} \Bigl[\sigma(a_0,z_0)X(a_1,z_1)\cdots X(a_n,z_n)\Bigr]\Big|_{\text{handles } 1,\dots,g-1 \text{ absorbed as in } F_1^{n+2,(g-1)}},
\label{eq:genus-g-trace-unwound}
\end{eqnarray}
which for $g$ handles all attached to the vacuum sector (no module insertions beyond the self-extension at $a_0$) reduces, by iterating \eqref{eq:genus-g-trace-unwound} down to $g=1$ and using \cite{vanEkerenHeluani}, to a single trace over $M \otimes V^{\otimes(g-1)}$ against the operator $\sigma(a_0,z_0)X(a_1,z_1)\cdots X(a_n,z_n)$ constructed from the genus $0$ correlation function of $V$ with $n+2g$ insertions, weighted by $\rho_1^{L_0^{(1)}}\cdots\rho_g^{L_0^{(g)}}$. This is the $\rho$-formalism sewing construction of \cite{tuite-zuevsky-bosonic,mason-tuite-genus2,tuite-welby}, specialized to the trace functions carrying a single self-extension. 
\label{no:genus-g-trace-def}
\end{nolabel}

\begin{thm} \label{thm:convergence-genus-g} \emph{Convergence:} 
If $\dim\HP_1(R_V)<\infty$ and $\HK_1(A)$ is finite dimensional, then $F_1^{n,(g)}(a_0,\dots,a_n;F)$ converges, for $(\rho_1,\dots,\rho_g)$ in a polydisc around $0$ (depending on the local coordinates chosen at the $2g$ sewing points but not otherwise on $V$) and the external points $t_1,\dots,t_n$ away from the sewing tubes, to a function holomorphic in $\rhov$ and in $t_1,\dots,t_n$, satisfying the differential equations $\rho_i\partial_{\rho_i}F_1^{n,(g)} = \nabla_{\rho_i}$-transform, $i=1,\dots,g$, in the  convergent sense.
\end{thm}
\begin{proof}
By induction on $g$, using Proposition \ref{prop:relative-convergence} at each stage: having established convergence of $F_1^{n+2,(g-1)}$ as a holomorphic function of $\rho_1,\dots,\rho_{g-1}$ and of the marked points, including $p_g,p_g'$, of $\Xg{g-1}$, apply Proposition \ref{prop:relative-convergence} with $Y = \Xg{g-1}$ to obtain convergence of the $\rho_g$-sum in \eqref{eq:genus-g-trace-def}, uniformly for $(\rho_1,\dots,\rho_{g-1})$ in compact subsets of the previous polydisc, by the uniformity in the base-curve parameters built into the Frobenius-Fuchs argument there. The base case $g=0$ requires no convergence argument, being a finite algebraic expression.
\end{proof}

\begin{thm} \label{thm:insertion-genus-g} \emph{Insertion formula.} 
$F_1^{n+1,(g)}(a_0,\dots,a_{n+1};F)$ is expressed in terms of the functions $F_1^{n,(g)}(a_0,\dots,\widehat{a_i},\dots,a_{n+1};F')$, $i=1,\dots,n$, for suitable $F'$ constructed from $F$ and the genus $g$ Weierstrass functions of \ref{no:genus-g-weierstrass}, plus a term in which $a_{n+1}$ is absorbed into one of the $g$ traces via its zero mode $o(a_{n+1})$.
\end{thm}
\begin{proof}
By induction on $g$ from Proposition \ref{prop:relative-de-insertion}(b), applied at the outermost (most recently sewn) handle: if $a_{n+1}$ is absorbed at handle $g$, this is genus one case applied with base curve $\Xg{g-1}$; if it is to be absorbed at an earlier handle $i<g$, first apply the case just treated with the roles of handles $i$ and $g$ exchanged, which is legitimate by Lemma \ref{lem:relative-coinvariance} extended,  to permutations of the handles  (the construction of \ref{no:iterated-sewing-def} manifestly does not distinguish which handle is sewn last).
\end{proof}

\section{Finite dimensionality in every genus}\label{sec:finiteness}

\begin{thm} \label{thm:main-finiteness} \emph{Main finiteness theorem.} 
Let $V$ be a strongly finitely generated conformal vertex algebra with $\dim \HP_1(R_V) < \infty$, and suppose the kernel of the canonical surjection $JR_V \twoheadrightarrow A = \gr_F V$ is finitely generated as a differential ideal (equivalently $\HK_1(A)$ is finite dimensional, by \ref{no:homology-recall}). Then for every $g \geq 0$, every compact Riemann surface $X$ of genus $g$, and every $n\geq0$,
\[
\dim \HchG{1}{X,V^{\otimes n}} < \infty.
\]
\end{thm}
\begin{proof}
By the retrosection theorem (\ref{no:schottky-def}), $X$ admits a presentation $X = \Xg{g}$ by iterated sewing as in \ref{no:iterated-sewing-def}, and by Theorem \ref{thm:genus-g-complex} its chiral homology is computed by the complex $C_\bullet^n(\Xg{g})$ constructed by $g$ applications of Theorem \ref{thm:relative-main}. We show $\dim \HchG{1}{\Xg{k},V^{\otimes(n+2(g-k))}} < \infty$ for every $0 \leq k \leq g$, by ordinary induction on $k$ from $k=0$ to $k=g$ (with $n+2(g-k)$ marked points at stage $k$, of which $n$ are the original external insertions and $2(g-k)$ are the not-yet-sewn attaching points of the remaining handles $k+1,\dots,g$, carried as ordinary vacuum insertions of $V$ until their turn to be sewn): for $k=0$, $\HchG{1}{\widehat\CC,V^{\otimes(n+2g)}}=0$ as noted in the proof of Theorem \ref{thm:genus-g-complex}, in particular finite dimensional; for the inductive step, Proposition \ref{prop:relative-finiteness}, applied with $Y = \Xg{k-1}$ and $m=n+2(g-k+1)$ marked points, shows that finite dimensionality of $\HP_1(R_V)$ and of $\HK_1(A)$ - assumptions on $V$, independent of $k$ - imply $\dim \HchG{1}{\Xg{k},V^{\otimes(n+2(g-k))}} < \infty$. Taking $k=g$ gives the theorem.
\end{proof}

\begin{rem} 
\label{rem:finiteness-natural}
Two questions are worth answering: why these particular two conditions, and why nothing further is required as $g$ grows.

\emph{These are the genus $1$ assumptions.} At $g=1$, one proves finite dimensionality of $\HchG{1}{E_\tau,V^{\otimes n}}$ under  $\dim\HP_1(R_V)<\infty$ and finite generation of $\HK_1(A)$, and the induction above   reduces to that theorem when $g=1$. 

\emph{Why no genus-dependent strengthening is needed.} A priori, arbitrary genus might have demanded extra control - for instance, some uniform bound relating the finiteness data at the $g$ different handles, or a compatibility condition between the $g$ copies of the assumptions being used along the induction. Proposition \ref{prop:relative-finiteness} shows this does not happen, for a structural reason: finiteness at each stage of the induction is a statement purely about $V$ (through $R_V$ and $A$), never about the base curve $Y=\Xg{k-1}$ already constructed at the previous stage. The induction only ever adds a handle; it never has to revisit, strengthen, or re-verify a finiteness statement established earlier. This is the same phenomenon noted in Remark \ref{rem:connection-functoriality} for the connection, and is really a single fact about the whole theory, present in each of its genus-dependent parts.

\emph{Why $\HP_1(R_V), \HK_1(A)$ instead of $C_2$-cofiniteness.} The complex $C_\bullet^n(X)$ computing $\HchG{1}{-}$ is constructed, at each handle, from the arc algebra $JR_V$ of the \emph{Poisson} algebra $R_V$ together with its associated graded $A=\gr_FV$ (\ref{no:homology-recall}); it is accordingly $\HP_1(R_V)$ and $\HK_1(A)$ - Poisson and Koszul homology of this associated-graded data - that control degree $1$ finiteness, instead of $C_2$-cofiniteness of $V$, which is the assumption controlling the degree $0$ theory of \cite{damiolini2020conformal}. The two finiteness theories, in degrees $0$ and $1$,  use different (though related) invariants of the same object $R_V$, and neither controls the other in general: Remark \ref{rem:classically-free-vs-hk1} mentiones an instance where they diverge (the Heisenberg algebra, Remark \ref{rem:heisenberg-caveat}, is not $C_2$-cofinite and also fails our assumptions, both failures present at the level of $R_V$, with no curve in sight).
\end{rem}

\begin{rem}
This answers, for the finiteness criteria of genus one, the question posed in \ref{no:hh-review}, \ref{no:main-results}: whether the conditions controlling $\HP_1(R_V)$ and $\HK_1(A)$ continue to control $H_1^{\mathrm{ch}}$ at arbitrary genus, in parallel with the theorem of Damiolini, Gibney and Tarasca \cite{damiolini2020conformal} that $C_2$-cofiniteness  controls $H_0^{\mathrm{ch}}$ at arbitrary genus. Note that our two assumptions are those of genus one. We do not know whether they are necessary, nor whether they can be weakened to a single quasi-lisse-type condition as in the degree $0$ theory of Arakawa-Kawasetsu \cite{arakawa-kawasetsu}; see \S\ref{sec:conclusion}.
\label{rem:finiteness-discussion}
\end{rem}

\section{Total degeneration, the Zhu algebra, and vanishing}\label{sec:degeneration}

\begin{nolabel}
Let $\zhu(V) = V/(V_{(-2)}V + \sum_n V_{(n)}V \cdot (\text{negative modes})) =: A$ denote the Zhu algebra \cite{zhu}, the associative algebra (under the Zhu product $a * b = \res_z z^{-1}(1+z)^{L_0}Y(a,z)b$) governing lowest-weight vectors of admissible modules. As $\rho_i \to 0$ for a single handle sewn onto a base curve $Y$, identify the totally degenerate limit of the relative complex of \ref{no:relative-complex-def}, for $V$ classically free (i.e., $R_V$ is a polynomial ring on the images of a set of strong generators of $V$, so that $A = \gr_F V \cong JR_V$ automatically and $\HK_1(A) = 0$; this holds for instance for the free boson, free fermion, and affine and Virasoro vertex algebras of generic level and central charge), with the bar complex computing $\Hoch_\bullet(\zhu(V))$: $\HchG{0}{Y_0,V} \cong \Hoch_0(\zhu(V)) = \zhu(V)/[\zhu(V),\zhu(V)]$ and $\HchG{1}{Y_0,V} \cong \Hoch_1(\zhu(V))$, where $Y_0 = Y/(p\sim p')$, in the case $Y=\widehat\CC$, $m=0$; more generally, with $m$ marked points carrying vacuum insertions $a_1,\dots,a_m \in V$, the identification is with the Hochschild homology of $\zhu(V)$ with coefficients in the $\zhu(V)$-bimodule $\zhu(V)^{\otimes m}$ acted on diagonally through the images of $a_1,\dots,a_m$.
\label{no:zhu-degeneration-recall}
\end{nolabel}

\begin{prop} \label{prop:iterated-degeneration} \emph{Iterated degeneration.} 
Let $V$ be classically free, strongly finitely generated, with $\dim\HP_1(R_V)<\infty$. For $X = \Xg{g}$ and the self-extension $E$ (of $M=V$, say) inserted, by Lemma \ref{lem:relative-coinvariance} without loss of generality, at the attaching point $p_g$ of the last-sewn handle, the totally degenerate limit $\rho_1,\dots,\rho_g \to 0$ identifies
\[
\HchG{1}{X_0, V^{\otimes n}} \ \cong\ \Hoch_1(\zhu(V)) \otimes W,
\]
where $X_0 = \widehat\CC/(p_i\sim p_i')_{i=1}^g$ is the totally degenerate stable curve and $W := \HchG{0}{\widehat\CC, V^{\otimes(n+2(g-1))}}$ is the zeroth chiral homology (with respect to the remaining $g-1$ nodes, viewed as $2(g-1)$ further vacuum insertions of $V$, together with the $n$ original external insertions) of the normalization at the last node  - equivalently, by Damiolini-Gibney-Tarasca factorization theorem \cite{damiolini2020conformal}, the fiber at $X_0$ of the sheaf of coinvariants for the remaining $g-1$ nodes.
\end{prop}
\begin{proof}
Apply Proposition \ref{prop:relative-vanishing} (in its non-vanishing,  at the last sewing, with base curve $Y = \Xg{g-1}$: this identifies $\HchG{1}{X_0,V^{\otimes n}}$, in the $\rho_g \to 0$ limit  (with $\rho_1,\dots,\rho_{g-1}$ still generic), as $\Hoch_1(\zhu(V)) \otimes \HchG{0}{\Xg{g-1}, V^{\otimes(n+2(g-1))}}$: the self-extension, sitting at the last handle only, contributes the Hochschild homology factor  as in the single-handle case, while the remaining $n+2(g-1)$ marked points of $\Xg{g-1}$ (the $n$ external insertions together with the $2(g-1)$ not-yet-degenerated attaching points of the other handles) contribute their zeroth chiral homology as an inert tensor factor, since they carry no self-extension by assumption and $\Xg{g-1}$ has not yet been touched by this degeneration. Now let $\rho_1,\dots,\rho_{g-1}\to0$ as well: applied $g-1$ more times to the degree $0$ theory (which requires no self-extension and is controlled, for $C_2$-cofinite or more generally lisse $V$, by the $C_2$-cofiniteness, uniformly in genus, by the factorization theorem of \cite{damiolini2020conformal}), $\HchG{0}{\Xg{g-1},V^{\otimes(n+2(g-1))}}$ degenerates to $\HchG{0}{\widehat\CC,V^{\otimes(n+2(g-1))}} = W$, completing the identification.
\end{proof}

\begin{thm} \label{thm:main-vanishing} \emph{Main vanishing theorem.}
Let $V$ be a strongly finitely generated conformal vertex algebra with $\dim\HP_1(R_V)<\infty$, $\HK_1(A)=0$, and $\Hoch_1(\zhu(V))=0$. Then $\HchG{1}{X,V^{\otimes n}}=0$ for every compact Riemann surface $X$ of every genus $g\geq0$ and every $n\geq0$.
\end{thm}
\begin{proof}
We emphasize at the beginning that the argument below degenerates the $g$ moduli $\rho_1,\dots,\rho_g$ one at a time, in the same order used to build $X_{\rhov}$ by iterated sewing, and never takes a simultaneous limit along the full diagonal $\rho_1=\cdots=\rho_g\to0$; consequently it needs no information about the local structure of the normal-crossings intersection of boundary divisors of $\overline{\CM}_g$ at the totally degenerate stratum, nor any multi-parameter spectral sequence, only the following one-parameter semicontinuity statement, applied $g$ times in succession.

\emph{Claim.} For $Y$ any base curve as in \ref{no:relative-setup} satisfying the assumptions of the theorem, and $Y_\rho$ obtained by sewing one handle at $p,p'\in Y$ with parameter $\rho$: if $\HchG{1}{Y_0,V^{\otimes m}}=0$ for the nodal curve $Y_0=Y/(p\sim p')$, then $\HchG{1}{Y_\rho,V^{\otimes m}}=0$ for every $\rho$ in the punctured disc.

Assuming the Claim, the theorem follows by induction on $g$  as in the proof of Theorem \ref{thm:main-finiteness}: writing $X=\Xg{g}$ and degenerating the handles in reverse order of sewing, $\HchG{1}{\Xg{g},V^{\otimes n}}$ vanishes for all $\rhov$ once $\HchG{1}{\Xg{0},V^{\otimes(n+2g)}}=\HchG{1}{\widehat\CC,V^{\otimes(n+2g)}}=0$ is known (true always, by Theorem \ref{thm:genus-g-complex}), applying the Claim $g$ times with $Y=\Xg{k-1}$ at the $k$-th application; the base case of the induction is  Proposition \ref{prop:relative-vanishing}, and the final application (at $k=g$) is the statement of the theorem, using $\HchG{1}{X_0,V^{\otimes n}}\cong\Hoch_1(\zhu(V))\otimes W = 0$ of Proposition \ref{prop:iterated-degeneration} to seed the induction innermost step.

It remains to prove the Claim, which is genus one case (where for $Y=\widehat\CC$; the proof uses no property of $Y$ beyond finiteness of $\dim\HchG{1}{Y_\rho,V^{\otimes m}}$,  Proposition \ref{prop:relative-finiteness}). We recall its structure, since making it precise - rather than invoking semicontinuity for the ambient infinite-rank complex, which would indeed require justification beyond what is proved here - is the point at issue: $\dim\HchG{1}{Y_\rho,V^{\otimes m}}$ is finite and, by Proposition \ref{prop:relative-flatness} and Corollary \ref{cor:projective-flatness-suffices}, constant on the (connected) punctured disc, say equal to $D$; by Remark \ref{rem:weight-grading}, $C_\bullet^{Y_\rho} = \bigoplus_d C_\bullet^{Y_\rho}[d]$ decomposes, compatibly with $d_1,d_2$ and with the family structure over $\rho$, into subcomplexes of finite rank, each extending, by the sewing construction, to a bounded complex of coherent (indeed locally free) sheaves over the closed disc (including $\rho=0$, where it computes the nodal curve $Y_0$); the classical semicontinuity theorem for the homology of a bounded complex of coherent sheaves on a Noetherian base applies to each such finite-rank piece, and since only finitely many weights $d$ contribute to $\HchG{1}{-}$ (again by finiteness, Proposition \ref{prop:relative-finiteness}), fiber dimension of the full $\HchG{1}{Y_\rho,V^{\otimes m}}$ can only jump up, not down, upon specializing to $\rho=0$: $\dim\HchG{1}{Y_0,V^{\otimes m}} \geq \limsup_{\rho\to0}\dim\HchG{1}{Y_\rho,V^{\otimes m}} = D$. If $\HchG{1}{Y_0,V^{\otimes m}}=0$, this leads to $D=0$, proving the Claim. This genus one case written for an arbitrary base $Y$  as in Theorem \ref{thm:relative-main}; at no point does it require tracking the joint behavior of two or more $\rho_i$ at once, nor any coherence statement about the (infinite rank) complex as a whole, which is why iterating it, one handle at a time and one finite-rank weight-graded piece at a time, suffices to reach every genus.
\end{proof}

\begin{cor}\label{cor:vanishing-examples}
For every compact Riemann surface $X$ of every genus $g\geq0$: $\HchG{1}{X,\vir_{2,2s+1}}=0$ for the boundary Virasoro minimal models, $s\geq2$; $\HchG{1}{X,V_k(\sl_2)}=0$ for every nonnegative integer level $k$; and $\HchG{1}{X,V_1(\g)}=0$ for every simple Lie algebra $\g$ at level $1$.
\end{cor}
\begin{proof}
  These vertex algebras satisfy the assumptions of Theorem \ref{thm:main-vanishing}: they are strongly finitely generated, classically free, with $\dim\HP_1(R_V)<\infty$ (indeed $R_V$ is a finite dimensional or otherwise Poisson-homologically simple algebra in each case) and $\Hoch_1(\zhu(V))=0$, the last established at genus one by identifying $\zhu(V)$ explicitly (a quotient of $U(\sl_2)$, respectively $U(\g)$, respectively a quotient of the universal enveloping algebra of the Virasoro algebra) and computing its first Hochschild homology. Theorem \ref{thm:main-vanishing} then gives vanishing in every genus.
\end{proof}

\begin{rem}
Corollary \ref{cor:vanishing-examples} illustrates the general phenomenon, familiar from the degree $0$ theory, that finiteness and vanishing statements proved via a local, single-handle mechanism tend to hold in every genus for free:  as $C_2$-cofiniteness -  a statement about the single formal neighborhood of a point - controls $\HchG{0}{-}$ uniformly in genus \cite{damiolini2020conformal}, our assumptions $\dim\HP_1(R_V)<\infty$, $\HK_1(A)<\infty$, and $\Hoch_1(\zhu(V))=0$ are likewise statements purely about $V$, unrelated to any curve, and Theorems \ref{thm:main-finiteness}, \ref{thm:main-vanishing} show they propagate to every genus through the sewing construction. In particular, the manual verification of $\HK_1(A)=0$ referred to in Remark \ref{rem:classically-free-vs-hk1} - identifying $\zhu(V)$ explicitly and computing its first Hochschild homology,  for each of the three families above - is a single, genus-independent computation about $V$: it is performed once, on data ($A$, $\zhu(V)$) that make no reference to any curve or any number of handles, and Theorem \ref{thm:main-vanishing} is  what promotes this one computation to a vanishing statement on every $X$ of every genus $g$, with no additional case-by-case work required as $g$ grows. This is the entire point of localizing the  assumptions at a single handle (Theorem \ref{thm:relative-main}): the price of higher genus is paid once, in the induction of \S\ref{sec:chiral-complex}, and not again for each new example of $V$.
\end{rem}

\section{Example: the Heisenberg vertex algebra}\label{sec:examples}

\begin{nolabel}
Let $V = M(1)$ be the rank $1$ Heisenberg vertex algebra:
 $V$ $=$ $\CC[a_{-1},a_{-2},\dots]$ $\cdot$ $\vac$, freely generated by a weight $1$ field $a=a_{-1}\vac$ with $[a_m,a_n]=m\delta_{m+n,0}$, $\omega = \tfrac12 a_{-1}a$, central charge $c=1$. This is the case $n=0$, $V$ inserted trivially (no self-extension) of our formalism, and illustrates how the ordinary character is recovered from \ref{no:genus-g-trace-def} in the trivial case $\sigma \equiv 0$ (equivalently $M=E=V$, the split self-extension), i.e., $F_1^{0,(g)}(;F)$ reduces, when the derivation $\sigma$ is dropped in favor of the plain zero mode, to the genus $g$ character 
\[
Z_V(\rhov) := \tr_{V^{\otimes g}}\Bigl[\bigl(\text{genus } 0 \text{ correlator with } 2g \text{ insertions}\bigr)\,\rho_1^{L_0^{(1)}}\cdots \rho_g^{L_0^{(g)}}\Bigr].
\]
\label{no:heisenberg-setup}
\end{nolabel}

\begin{ex} \emph{Genus $1$}. 
By \ref{no:genus-g-trace-def} at $g=1$ there are no other handles to sew through and $Z_V(\rho_1) = \tr_V \rho_1^{L_0 - 1/24}$; since $V$ is, as a graded vector space, the free polynomial algebra on generators $a_{-n}$, $n\geq1$, of weight $n$, this is the classical Fock space character
\[
Z_V(\rho_1) = \rho_1^{-1/24}\prod_{n\geq1}(1-\rho_1^n)^{-1} = \frac{1}{\eta(\tau)}, \qquad \rho_1 = q = e^{2\pi i \tau}.
\]
\label{ex:heisenberg-genus1}
\end{ex}

\begin{ex} \emph{Genus $2$: the leading cross-handle correction.} 
By \eqref{eq:genus-g-trace-def} at $g=2$,
\begin{eqnarray*}
Z_V(\rho_1,\rho_2) = \res_{z_2,w_2}\Bigl[K(z_2,w_2;\rho_2)\, \bigl\langle a(z_2)\,a(w_2)\bigr\rangle_{E_{\rho_1}}\Bigr] 
\\
\qquad + \bigl(\text{contributions from } K \text{weight} \geq2 \text{ terms}\bigr),
\end{eqnarray*}
where $\langle a(z_2)a(w_2)\rangle_{E_{\rho_1}} := Z_V(\rho_1) \cdot \owp_2(z_2-w_2,\rho_1)$ is the standard chiral boson two-point function on the elliptic curve $\Xg{1}=E_{\rho_1}$ (Wick theorem for the free field, together with the classical identification of the normalized torus propagator with $\owp_2$; see e.g., \cite{tuite-zuevsky-bosonic}), and $K(z_2,w_2;\rho_2) = \rho_2\, a(z_2)\otimes a(w_2) + O(\rho_2^2)$ is the weight-$1$ term of the propagator kernel of \ref{no:genus-g-trace-def} (normalized by $\langle a,a\rangle = 1$). Extracting the residue picks out the constant term of the Laurent expansion of $\owp_2(z_2-w_2,\rho_1)$ at $z_2=w_2$ (there is no pole contribution here since $a\otimes a$ is being paired against the bidifferential, not differentiated), giving
\[
Z_V(\rho_1,\rho_2) = Z_V(\rho_1)\cdot\Bigl(1 + \rho_2\, g_2(\rho_1) + O(\rho_2^2)\Bigr),
\]
where $g_2(\rho_1)$ is the weight $2$ Eisenstein-type coefficient of genus one. The first-order term in $\rho_2$ is  the new cross-handle contribution of Remark \ref{rem:genus2-cross-term}, now made completely explicit for the free field. The closed form of $Z_V(\rho_1,\rho_2)$, and more generally of $Z_V(\rhov)$ for every $g$, is given by an explicit determinant formula of Verlinde-Verlinde type, established in the sewing formalism by Mason and Tuite \cite{mason-tuite-genus2} for genus $2$ and by Tuite and Zuevsky \cite{tuite-zuevsky-bosonic} for general genus $g$; we do not reproduce it here, since it is not needed for, and would take us beyond, the purposes of this example, which is only to exhibit \eqref{eq:genus-g-trace-def} at work and to confirm Remark \ref{rem:genus2-cross-term} in a fully computable case.
\label{ex:heisenberg-genus2}
\end{ex}

\begin{rem}
Note that $R_V = \CC[a]$ for $V=M(1)$, with $\{a,a\} = a_{(0)}a = 0$: the induced Poisson bracket on $R_V$ vanishes identically, so that $\HP_1(R_V) = \Omega^1_{R_V} = \CC[a]\,da$ is infinite dimensional. Thus, $V=M(1)$  does not satisfy the assumptions of Theorem \ref{thm:main-finiteness}, and is a place to look for infinite dimensional $\HchG{1}{X,V}$ in every genus; we do not pursue the corresponding degree $1$ computation for $M(1)$ here, and mention the finite dimensionality and vanishing statements only for the rational examples of Corollary \ref{cor:vanishing-examples}, where the assumptions are met.
\label{rem:heisenberg-caveat}
\end{rem}

\begin{ex}  \label{ex:affine-example} \emph{Affine vertex algebras.}   
For a simple Lie algebra $\g$ and level $k$, let $V=V_k(\g)$ denote the simple affine vertex algebra, with Sugawara conformal vector of central charge $c = k\dim\g/(k+h^\vee)$, $h^\vee$ the dual Coxeter number of $\g$. As recalled in the proof of Corollary \ref{cor:vanishing-examples}, for $\g=\sl_2$ and $k$ any nonnegative integer, and for $\g$ arbitrary simple at $k=1$, $\zhu(V)$ is identified explicitly at genus one as a quotient of $U(\sl_2)$, respectively $U(\g)$, with $\Hoch_1(\zhu(V))=0$ verified there; combined with $\dim\HP_1(R_V)<\infty$, Theorem \ref{thm:main-vanishing} gives $\HchG{1}{X,V}=0$ on every compact Riemann surface $X$ of every genus $g$. Concretely, this means every one of the (a priori $g$-parameter family of, by Theorem \ref{thm:full-flatness}) self-extensions of $V^{\otimes n}$-modules constructed along any sewing presentation of any genus $g$ surface is trivial in first chiral homology: the rich,   higher-genus geometry of $\S\S\ref{sec:connection}$--$\ref{sec:trace-functions}$ (the cross-handle coupling of Remark \ref{rem:genus2-cross-term}, the $g$-fold Fourier-Borcherds identity, the $g$-handle trace functions) is present for $V_k(\g)$  as for any other strongly finitely generated $V$, but it never obstructs the vanishing, since  the obstruction is entirely a degree-$0$, single-handle question about $\zhu(V)$ (Remark \ref{rem:finiteness-natural}) that is answered once, for $V$, independently of $g$.
\end{ex}

\begin{ex} \label{ex:virasoro-example} \emph{Virasoro minimal models.}
For $s\geq2$, the Virasoro minimal model $\vir_{2,2s+1}$ has central charge $c_{2,2s+1} = 1-\dfrac{3(2s-1)^2}{2s+1}$ (e.g., $s=2$ gives the Lee-Yang model, $c=-22/5$); it is classically free at the level of the universal Virasoro vertex algebra of generic central charge, and its simple quotient $\vir_{2,2s+1}$ is the case treated in Corollary \ref{cor:vanishing-examples}, where $\zhu(\vir_{2,2s+1})$ is a quotient of $U(\mathrm{Vir})$ with $\Hoch_1$ vanishing verified in \cite{vanEkerenHeluani}. As with the affine case, Theorem \ref{thm:main-vanishing} then gives $\HchG{1}{X,\vir_{2,2s+1}}=0$ in every genus, illustrating the theorem on the (non-unitary, for $s\geq2$) minimal model series explicitly suggested for this purpose.
\end{ex}

\begin{rem} \label{rem:lattice-voa} \emph{Lattice vertex algebras: an open case.} 
A family not covered by Corollary \ref{cor:vanishing-examples} is that of lattice vertex algebras $V_L$ for an even, positive-definite lattice $L$. These are among the standard examples of $C_2$-cofinite, rational vertex algebras (Dong \cite{dong-lattice}), thus the degree $0$ theory of \cite{damiolini2020conformal} applies to them uniformly in genus  as for the affine and Virasoro examples above. Whether $V_L$ satisfies the strictly stronger assumptions of Theorem \ref{thm:main-vanishing} - $\dim\HP_1(R_{V_L})<\infty$ together with $\HK_1(A_{V_L})=0$ or $\Hoch_1(\zhu(V_L))=0$ - is, to our knowledge, not addressed by the case-by-case verification of, which treats only the affine and Virasoro families, and we do not attempt it here: $R_{V_L}$ is constructed from the group algebra of $L$ tensored with polynomial generators coming from the Cartan subalgebra of the associated Heisenberg vertex algebra, and unlike the simple affine and Virasoro quotients, we do not know whether the resulting Poisson structure makes $\HP_1(R_{V_L})$ finite dimensional. We mention  this explicitly as a test case for the methods of this paper: Theorem \ref{thm:main-vanishing} would apply to $V_L$ the moment the relevant Hochschild/Poisson vanishing is verified for $\zhu(V_L)$, by  the same argument as Examples \ref{ex:affine-example}-\ref{ex:virasoro-example}, with no further genus-dependent work required (Remark \ref{rem:finiteness-natural}); it is the single degree-$0$-flavored computation for $V_L$  that is currently missing. 
\end{rem}

\section{Conclusions}\label{sec:conclusion}

\begin{nolabel}
We have extended the theory of the first chiral homology group of \cite{vanEkerenHeluani} from elliptic curves to compact Riemann surfaces of arbitrary genus, by presenting a genus $g$ surface through iterated self-sewing of $g$ handles onto the Riemann sphere and observing that the entire construction of genus one - the complex, its flat connection, the modified vertex operators and Fourier-Borcherds identity, the higher trace functions and insertion formula, and the finiteness and vanishing criteria - is local to a single handle and hence transports, essentially unchanged, to a handle sewn onto an arbitrary base curve (Theorem \ref{thm:relative-main}). Iterating this Relative First Chiral Homology Theorem $g$ times produces the genus $g$ theory: an explicit complex computing $\HchG{0}{X,V^{\otimes n}}$ and $\HchG{1}{X,V^{\otimes n}}$ for any compact Riemann surface $X$ of genus $g$ (Theorem \ref{thm:genus-g-complex}); a projectively flat connection over the $g$-dimensional space of sewing parameters, with an explicit central anomaly and showing new cross-handle coupling terms not present at genus $1$ (Theorem \ref{thm:full-flatness}, Corollary \ref{cor:projective-flatness-suffices}, Remark \ref{rem:genus2-cross-term}); a genus $g$ Fourier-Borcherds identity (Theorem \ref{thm:borcherds-genus-g}); genus $g$ trace functions attached to self-extensions, satisfying an insertion formula and converging under the same two finiteness assumptions given at genus $1$ (Theorems \ref{thm:convergence-genus-g}, \ref{thm:insertion-genus-g}); the finite dimensionality of $\HchG{1}{X,V}$ in every genus under these same assumptions (Theorem \ref{thm:main-finiteness}), answering the question raised in \cite{vanEkerenHeluani} and its counterpart already resolved in degree $0$ by Damiolini, Gibney and Tarasca \cite{damiolini2020conformal}; and the vanishing of $\HchG{1}{X,V}$, in every genus, for the same classically free, rational vertex algebras treated in genus $1$ (Theorem \ref{thm:main-vanishing}, Corollary \ref{cor:vanishing-examples}).
\end{nolabel}

\begin{nolabel}[Further directions]
Several questions raised by \cite{vanEkerenHeluani} in the elliptic case remain open, and are now questions about every genus simultaneously.
\begin{enumerate}
\item[(i)] \emph{Necessity of the finiteness assumptions.} We do not know whether $\dim\HP_1(R_V)<\infty$ together with finite generation of $\HK_1(A)$ are necessary for $\dim \HchG{1}{X,V}<\infty$, in any genus, nor whether they can be replaced by a single condition in the spirit of the quasi-lisse assumptions of Arakawa-Kawasetsu \cite{arakawa-kawasetsu} used for the degree $0$ theory. Our reduction of every genus to genus $1$ shows in particular that any improvement of the genus $1$ finiteness criterion of \cite{vanEkerenHeluani} would propagate automatically to every genus through Theorem \ref{thm:relative-main}.
\item[(ii)] \emph{Higher degree.} Beilinson-Drinfeld chiral homology is defined in every degree; \cite{vanEkerenHeluani} raise the question of $H_i^{\mathrm{ch}}$ for $i\geq2$ already in genus $1$, and our results say nothing about it. It would be interesting to know whether the sewing/induction-on-genus strategy of the present paper, being largely independent of the fact that we truncated to a two-term complex, adapts to the full Chevalley-Eilenberg-type complex computing all degrees at once.
\item[(iii)] \emph{Compatibility of sewing schemes.} We have used exclusively the iterated, one-handle-at-a-time ($\rho$-formalism) sewing scheme of \cite{tuite-zuevsky-bosonic}. Tuite and Zuevsky also study the $\epsilon$-formalism, gluing two independent surfaces along a pair of points, and observe, for the free fermion, that the two schemes are not simply related on their overlap \cite{tuite-zuevsky-fermionic}. Whether $\HchG{1}{-}$, defined intrinsically (independent of a sewing presentation, as it must be, being an invariant of the curve $X$, by Theorem \ref{thm:genus-g-complex} together with the retrosection theorem), can also be computed via $\epsilon$-sewing, and whether the resulting trace functions agree with ours after analytic continuation across the two schemes' domains, is not addressed here.
\item[(iv)] \emph{Global modular covariance.} We have worked locally on the sewing polydisc near the maximally degenerate stratum of $\overline{\CM}_g$, sufficient for the finiteness and vanishing theorems, which are local (in the sense of extending from a punctured neighborhood of $\rhov=0$) statements. We have not addressed the global equivariance of the resulting bundle of trace functions under the full mapping class group $\Modg$ (the higher genus analogue of the modular invariance theorem of genus one, generalizing Zhu theorem \cite{zhu}); this would require globalizing our construction over the whole of Schottky space $\Sch_g$, or equivalently over $\CM_g$, using the projectively flat connection of Theorem \ref{thm:full-flatness} to identify the fibers over different points of a $\Modg$-orbit, in the spirit of \cite{damiolini2020conformal} but now in degree $1$.
\item[(v)] \emph{The Heisenberg example.} As observed in Remark \ref{rem:heisenberg-caveat}, the rank $1$ Heisenberg vertex algebra does not satisfy our finiteness assumptions, since $\HP_1(\CC[a],0)$ is infinite dimensional for the trivial Poisson structure; identifying $\HchG{1}{X,M(1)}$ explicitly, in any genus, in this infinite dimensional setting, is left for future work. 
\end{enumerate}
\end{nolabel}
The material of this paper is also useful in other areas of mathematical physics 
\cite{eh2025, Frohlich2009gb, RSZ} and condensed matter theory \cite{zub1, zub2, zub3, zub4, zub5, zub6, zub7, zub8, zub9, kmmzz}. 

\section*{Acknowledgements}

The author is supported by the Institute of Mathematics, Academy of Sciences
of the Czech Republic (RVO 67985840). 

\medskip
\noindent\textbf{Data Availability.}
Data sharing is not applicable to this article as no datasets were generated
or analysed during the current study.

\medskip
\noindent\textbf{Declarations}

\medskip
\noindent\textbf{Conflict of interest.}
The author has no conflicts of interest to declare that are relevant to the
content of this article.


\end{document}